\documentclass{amsart}

\usepackage[utf8]{inputenc}

\usepackage[english]{babel}

\usepackage{amsmath}
\usepackage{amsfonts}
\usepackage{amssymb}
\usepackage{amsthm}

\usepackage{color}
\usepackage{tikz}
\usepackage{pgfplots}
\usetikzlibrary{calc}
\usetikzlibrary{shapes}

\usepackage{caption}
\usepackage{url} 
\usepackage{enumerate}

\usepackage[hyperfootnotes]{hyperref}

\theoremstyle{plain} 
\newtheorem{theo}{Theorem}[section]
\newtheorem*{theo*}{Theorem} 
\newtheorem{coro}[theo]{Corollary}
\newtheorem{prop}[theo]{Proposition}
\newtheorem{lemma}[theo]{Lemma}
\newtheorem{conj}[theo]{Conjecture}
\theoremstyle{definition}
\newtheorem{defi}[theo]{Definition}

\newtheorem{remark}[theo]{Remark}

\newtheorem{ex}[theo]{Example}

\numberwithin{equation}{section}

\newcommand{\eps}{\varepsilon}
\newcommand{\DD}{\Delta}
\newcommand{\h}{\widehat}
\newcommand{\pp}{\Phi^{{\scriptscriptstyle +}}}
\newcommand{\hpp}{\h{\Phi}}

\newcommand{\OO}{\operatorname{O}\nolimits}
\newcommand{\GL}{\operatorname{GL}\nolimits}
\newcommand{\rk}{\operatorname{rk}\nolimits}

\newcommand{\dist}{\operatorname{dist}\nolimits}
\newcommand{\dep}{\operatorname{dp}\nolimits}
\newcommand{\oo}{\infty}
\newcommand{\cone}{\operatorname{cone}}
\newcommand{\conv}{\operatorname{conv}}
\newcommand{\Span}{\operatorname{span}}
\newcommand{\Vp}{V^{\perp}}

\author[C. Hohlweg]{Christophe~Hohlweg}
\address[Christophe Hohlweg]{Universit\'e du Qu\'ebec \`a Montr\'eal\\
LaCIM et D\'epartement de Math\'ematiques\\ CP 8888 Succ. Centre-Ville\\
Montr\'eal, Qu\'ebec, H3C 3P8\\ Canada}
\email{hohlweg.christophe@uqam.ca}
\urladdr{http://hohlweg.math.uqam.ca}

\author[J.~P. Labb\'e]{Jean-Philippe~Labb\'e}
\address[Jean-Philippe Labb\'e]{Freie Universit\"at Berlin\\ Institut f\"ur Mathematik\\
   Arnimallee~2 \\14195, Berlin, Deutschland}
\email{labbe@math.fu-berlin.de}
\urladdr{http://page.mi.fu-berlin.de/labbe}

\author[V. Ripoll]{Vivien~Ripoll}
\address[Vivien Ripoll]{Universit\'e du Qu\'ebec \`a Montr\'eal\\
LaCIM et D\'epartement de Math\'ematiques\\ CP 8888 Succ. Centre-Ville\\
Montr\'eal, Qu\'ebec, H3C 3P8\\ Canada}
\email{vivien.ripoll@lacim.ca}
\urladdr{http://www.math.ens.fr/~vripoll}

\thanks{The first author is supported by a NSERC grant, the second
  author is supported by a FQRNT doctoral scholarship and the third
  author is supported by a postdoctoral fellowship from CRM-ISM and
  LaCIM}


\begin{document}

\title[Asymptotical behaviour of roots of infinite Coxeter groups]{Asymptotical behaviour of roots\\ of infinite Coxeter groups}

%
%

\begin{abstract}
Let $W$ be an infinite Coxeter group. We initiate the study of the set
$E$ of limit points of ``normalized'' roots (representing the
directions of the roots) of W. We show that $E$ is contained in the
isotropic cone $Q$ of the bilinear form $B$ associated to a geometric
representation, and illustrate this property with numerous examples
and pictures in rank $3$ and $4$. We also define a natural geometric
action of $W$ on $E$, and then we exhibit a countable subset of $E$,
formed by limit points for the dihedral reflection subgroups of
$W$. We explain how this subset is built from the intersection
with $Q$ of the lines passing through two positive roots, and finally we
establish that it is dense in~$E$.

\end{abstract}

\maketitle

%
%

\section*{Introduction} \label{sec:intro}

When dealing with Coxeter groups, one of the most powerful tools we have at our disposal is the notion of root systems. In the
case of a finite Coxeter group $W$, i.e., a finite reflection group, roots are representatives of normal vectors for the Euclidean
reflections in $W$. Thinking about finite Coxeter groups and their associated \emph{finite} root systems allows the use of
arguments from Euclidean geometry and finite group theory, which makes finite root systems well studied, see for instance
\cite[Ch.1]{humphreys}, and the references therein.

To deal with root systems of infinite Coxeter groups, we usually distinguish two classes: affine reflection groups, and the other
infinite but not affine Coxeter groups. The  root systems associated to affine Coxeter groups are also well-studied: an affine root
system can be realized in an affine Euclidean space as a finite root system up to translations, see for instance
\cite[Ch.4]{humphreys}. For the other \emph{infinite} (not affine) Coxeter groups, in comparison, very little is known.

While investigating a conjecture on biclosed sets of positive roots in
an infinite root system (Conjecture 2.5 in \cite{dyer:weak}), we came
across a difficulty: we do not know much how the roots of an infinite
root system are geometrically distributed over the space.

First, observe that even the term \emph{(infinite) root system} seems to designate different objects, depending on whether
associated to Lie algebras (see \cite{kac,loos-neher}), Kac-Moody Lie algebras (see \cite{moody-pianzola}) or Coxeter groups
via their geometric representations (see \cite[Ch.5 \& 6]{humphreys}). While all definitions of root systems are related to a given
bilinear form, the bilinear forms considered in the case of Kac-Moody algebras or Lie algebras are different from the one in the
classical definition of a root system for infinite Coxeter groups. In particular, a difference lies in the ability to change the value of
the bilinear form on a pair of reflections whose product has infinite order. In this vein, more general geometric representations of
a Coxeter group and of root systems (that we take up in \S\ref{sec:rep}) have been introduced. These more general geometric
representations were recently presented in \cite{krammer:conjugacy} and \cite{bonnafe-dyer} (see also Howlett
\cite{howlett:coxeter}) but seem to go back to E.~B.~Vinberg~\cite{vinberg}, as stated in D.~Krammer's thesis. They have been the
framework of several recent studies about infinite root systems of Coxeter groups (see for instance
\cite{bonnafe-dyer,dyer:rigidity,dyer:weak,fu1,fu2}).

\medskip

Taking up this framework and using the computer algebra system Sage, we obtain the following pictures (Figures~\ref{fig:sage}(a)
and \ref{fig:sage}(b)), which suggests that roots have a very interesting asymptotical behaviour. In this article, we initiate the study
of this behaviour.

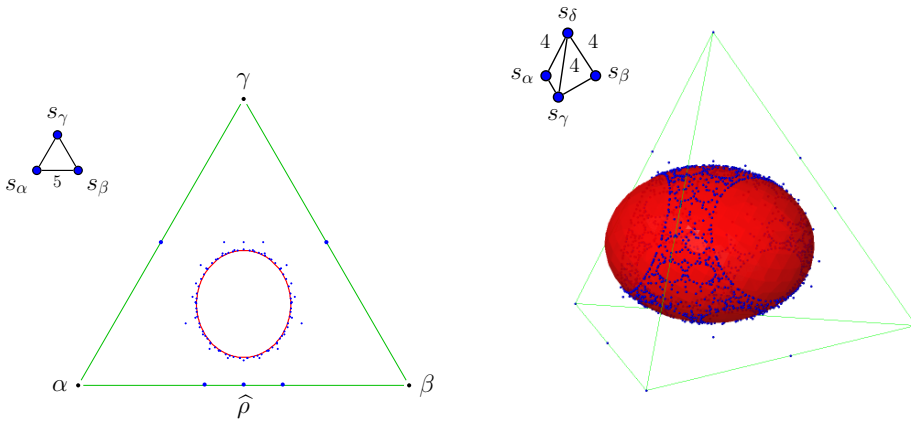
\begin{figure}[!ht]
\centering
\captionsetup{width=0.9\textwidth}
\begin{tabular}{p{.45\linewidth}@{\hspace{1cm}}p{.45\linewidth}}

\scalebox{0.55}{
\begin{tikzpicture}
	[scale=2,
	 q/.style={red,line join=round,thick},
	 racine/.style={blue},
	 racinesimple/.style={black},
	 racinedih/.style={blue},
	 sommet/.style={inner sep=2pt,circle,draw=black,fill=blue,thick,anchor=base},
	 rotate=0]

\def\grosseur{0.0125}
\def\grosseursimple{0.025}

\def\grosseurdih{0.0075}


\draw[q] (2.57,0.99) --
(2.57,1.0) --
(2.57,1.01) --
(2.57,1.02) --
(2.56,1.08) --
(2.56,1.09) --
(2.56,1.1) --
(2.56,1.11) --
(2.55,1.14) --
(2.55,1.15) --
(2.54,1.18) --
(2.54,1.19) --
(2.53,1.22) --
(2.52,1.24) --
(2.52,1.25) --
(2.51,1.27) --
(2.5,1.29) --
(2.49,1.31) --
(2.48,1.33) --
(2.47,1.35) --
(2.45,1.38) --
(2.44,1.39) --
(2.43,1.41) --
(2.42,1.42) --
(2.41,1.43) --
(2.34,1.5) --
(2.33,1.51) --
(2.32,1.52) --
(2.31,1.53) --
(2.29,1.54) --
(2.28,1.55) --
(2.26,1.56) --
(2.24,1.57) --
(2.22,1.58) --
(2.2,1.59) --
(2.18,1.6) --
(2.17,1.6) --
(2.15,1.61) --
(2.14,1.61) --
(2.11,1.62) --
(2.1,1.62) --
(2.06,1.63) --
(2.05,1.63) --
(2.04,1.63) --
(2.03,1.63) --
(2.02,1.63) --
(2.01,1.63) --
(1.99,1.63) --
(1.98,1.63) --
(1.97,1.63) --
(1.96,1.63) --
(1.95,1.63) --
(1.94,1.63) --
(1.9,1.62) --
(1.89,1.62) --
(1.86,1.61) --
(1.85,1.61) --
(1.83,1.6) --
(1.82,1.6) --
(1.8,1.59) --
(1.78,1.58) --
(1.76,1.57) --
(1.74,1.56) --
(1.72,1.55) --
(1.71,1.54) --
(1.69,1.53) --
(1.68,1.52) --
(1.67,1.51) --
(1.66,1.5) --
(1.59,1.43) --
(1.58,1.42) --
(1.57,1.41) --
(1.56,1.39) --
(1.55,1.38) --
(1.53,1.35) --
(1.52,1.33) --
(1.51,1.31) --
(1.5,1.29) --
(1.49,1.27) --
(1.48,1.25) --
(1.48,1.24) --
(1.47,1.22) --
(1.46,1.19) --
(1.46,1.18) --
(1.45,1.15) --
(1.45,1.14) --
(1.44,1.11) --
(1.44,1.1) --
(1.44,1.09) --
(1.44,1.08) --
(1.43,1.02) --
(1.43,1.01) --
(1.43,1.0) --
(1.43,0.99) --
(1.43,0.98) --
(1.43,0.97) --
(1.43,0.96) --
(1.43,0.95) --
(1.44,0.88) --
(1.44,0.87) --
(1.44,0.86) --
(1.45,0.83) --
(1.45,0.82) --
(1.45,0.81) --
(1.46,0.79) --
(1.46,0.78) --
(1.47,0.75) --
(1.48,0.72) --
(1.49,0.7) --
(1.5,0.68) --
(1.51,0.66) --
(1.52,0.64) --
(1.53,0.62) --
(1.54,0.6) --
(1.55,0.59) --
(1.57,0.56) --
(1.58,0.55) --
(1.59,0.54) --
(1.61,0.51) --
(1.62,0.5) --
(1.63,0.49) --
(1.67,0.46) --
(1.68,0.45) --
(1.69,0.44) --
(1.71,0.43) --
(1.72,0.42) --
(1.74,0.41) --
(1.76,0.4) --
(1.77,0.39) --
(1.78,0.39) --
(1.8,0.38) --
(1.82,0.37) --
(1.85,0.36) --
(1.86,0.36) --
(1.88,0.35) --
(1.89,0.35) --
(1.9,0.35) --
(1.93,0.34) --
(1.94,0.34) --
(1.95,0.34) --
(1.96,0.34) --
(1.97,0.34) --
(1.98,0.34) --
(1.99,0.34) --
(2.01,0.34) --
(2.02,0.34) --
(2.03,0.34) --
(2.04,0.34) --
(2.05,0.34) --
(2.06,0.34) --
(2.07,0.34) --
(2.1,0.35) --
(2.11,0.35) --
(2.12,0.35) --
(2.14,0.36) --
(2.15,0.36) --
(2.18,0.37) --
(2.2,0.38) --
(2.22,0.39) --
(2.23,0.39) --
(2.24,0.4) --
(2.26,0.41) --
(2.28,0.42) --
(2.29,0.43) --
(2.31,0.44) --
(2.32,0.45) --
(2.33,0.46) --
(2.37,0.49) --
(2.38,0.5) --
(2.39,0.51) --
(2.41,0.54) --
(2.42,0.55) --
(2.43,0.56) --
(2.45,0.59) --
(2.46,0.6) --
(2.47,0.62) --
(2.48,0.64) --
(2.49,0.66) --
(2.5,0.68) --
(2.51,0.7) --
(2.52,0.72) --
(2.53,0.75) --
(2.54,0.78) --
(2.54,0.79) --
(2.55,0.81) --
(2.55,0.82) --
(2.55,0.83) --
(2.56,0.86) --
(2.56,0.87) --
(2.56,0.88) --
(2.57,0.95) --
(2.57,0.96) --
(2.57,0.97) --
(2.57,0.98) --
cycle;

\node[label=left :{{\huge $\alpha$}}] (a) at (0.000000000000000,0.000000000000000) {};
\fill[racinesimple] (0.000000000000000,0.000000000000000) circle (\grosseursimple);\node[label=right :{{\huge$\beta$}}] (b) at (4.00000000000000,0.000000000000000) {};
\fill[racinesimple] (4.00000000000000,0.000000000000000) circle (\grosseursimple);\node[label=above :{{\huge$\gamma$}}] (g) at (2.00000000000000,3.46410161513775) {};
\fill[racinesimple] (2.00000000000000,3.46410161513775) circle (\grosseursimple);
\draw[green!75!black] (a) -- (b) -- (g) -- (a);

\fill[racine] (1.52786404500042,0.01) circle (\grosseursimple);
\fill[racine] (3.00000000000000,1.73205080756888) circle (\grosseursimple);
\fill[racine] (2.47213595499958,0.01) circle (\grosseursimple);
\fill[racine] (1.00000000000000,1.73205080756888) circle (\grosseursimple);

\fill[racine] (1.29925418795460,0.750124755161338) circle (\grosseur);
\fill[racine] (2.70074581204540,0.750124755161338) circle (\grosseur);
\fill[racine] (2.23606797749979,1.73205080756888) circle (\grosseur);
\fill[racine] (1.76393202250021,1.73205080756888) circle (\grosseur);
\fill[racine] (2.00000000000000,0.01) circle (\grosseursimple);

\node[label=below :{{\huge$\h{\rho}$}}] (a) at (2.00000000000000,0.01) {};

\fill[racine] (2.00000000000000,1.73205080756888) circle (\grosseur);
\fill[racine] (1.48107236977278,1.45430354538087) circle (\grosseur);
\fill[racine] (1.58797734083340,0.441056358833072) circle (\grosseur);
\fill[racine] (1.38196601125011,1.07046626931927) circle (\grosseur);
\fill[racine] (2.51892763022721,1.45430354538087) circle (\grosseur);
\fill[racine] (2.61803398874989,1.07046626931927) circle (\grosseur);
\fill[racine] (2.41202265916660,0.441056358833072) circle (\grosseur);

\fill[racine] (2.58359213500126,0.817763462139324) circle (\grosseur);
\fill[racine] (1.41640786499874,0.817763462139324) circle (\grosseur);
\fill[racine] (2.49197495565797,0.616603890626973) circle (\grosseur);
\fill[racine] (1.80901699437495,0.330792269124804) circle (\grosseur);
\fill[racine] (2.25464400750007,1.59575689721232) circle (\grosseur);
\fill[racine] (1.74535599249993,1.59575689721232) circle (\grosseur);
\fill[racine] (1.42398568894741,1.23320778125395) circle (\grosseur);
\fill[racine] (2.57601431105259,1.23320778125395) circle (\grosseur);
\fill[racine] (1.50802504434203,0.616603890626973) circle (\grosseur);
\fill[racine] (2.19098300562505,0.330792269124804) circle (\grosseur);

\fill[racine] (1.69984026822590,0.420601387045539) circle (\grosseur);
\fill[racine] (1.59674775249769,0.505405614359890) circle (\grosseur);
\fill[racine] (1.46864077564517,0.703078094178913) circle (\grosseur);
\fill[racine] (2.30015973177410,0.420601387045539) circle (\grosseur);
\fill[racine] (2.33558878217749,1.52175011404611) circle (\grosseur);
\fill[racine] (2.00000000000000,0.301957859349469) circle (\grosseur);
\fill[racine] (2.11145618000168,1.63552692427865) circle (\grosseur);
\fill[racine] (1.66441121782251,1.52175011404611) circle (\grosseur);
\fill[racine] (2.58017872829546,1.00489903487844) circle (\grosseur);
\fill[racine] (1.41982127170454,1.00489903487844) circle (\grosseur);
\fill[racine] (2.53135922435483,0.703078094178913) circle (\grosseur);
\fill[racine] (2.45176800710533,1.41552891889277) circle (\grosseur);
\fill[racine] (1.54823199289467,1.41552891889277) circle (\grosseur);
\fill[racine] (1.88854381999832,1.63552692427865) circle (\grosseur);
\fill[racine] (2.40325224750231,0.505405614359890) circle (\grosseur);

\fill[racine] (2.38196601125011,1.47934800038893) circle (\grosseur);
\fill[racine] (2.00000000000000,1.64520628321698) circle (\grosseur);
\fill[racine] (2.34698535516948,0.459120291958359) circle (\grosseur);
\fill[racine] (1.74188649750014,1.57030091762589) circle (\grosseur);
\fill[racine] (1.61803398874989,1.47934800038893) circle (\grosseur);
\fill[racine] (1.42784099282795,0.930308113421365) circle (\grosseur);
\fill[racine] (1.43428069454515,1.11098904875203) circle (\grosseur);
\fill[racine] (2.21654236465910,0.375062377580669) circle (\grosseur);
\fill[racine] (2.12250219613650,0.343314496628835) circle (\grosseur);
\fill[racine] (2.56571930545485,1.11098904875203) circle (\grosseur);
\fill[racine] (2.56143409826268,0.830556410615058) circle (\grosseur);
\fill[racine] (2.46352549156242,0.598408836454621) circle (\grosseur);
\fill[racine] (1.48644565634144,1.28729969973672) circle (\grosseur);
\fill[racine] (2.25811350249986,1.57030091762589) circle (\grosseur);
\fill[racine] (2.57215900717204,0.930308113421365) circle (\grosseur);
\fill[racine] (2.18349992631140,1.60348538829502) circle (\grosseur);
\fill[racine] (2.51355434365856,1.28729969973672) circle (\grosseur);
\fill[racine] (1.87749780386350,0.343314496628835) circle (\grosseur);
\fill[racine] (1.43856590173732,0.830556410615058) circle (\grosseur);
\fill[racine] (1.65301464483052,0.459120291958359) circle (\grosseur);
\fill[racine] (1.81650007368860,1.60348538829502) circle (\grosseur);
\fill[racine] (1.53647450843758,0.598408836454621) circle (\grosseur);
\fill[racine] (1.78345763534090,0.375062377580669) circle (\grosseur);

\fill[racine] (2.04508497187474,0.330792269124804) circle (\grosseur);
\fill[racine] (1.43222828170721,0.889502215664309) circle (\grosseur);
\fill[racine] (2.57294901687516,0.992376807374411) circle (\grosseur);
\fill[racine] (2.30015973177410,0.420601387045538) circle (\grosseur);
\fill[racine] (1.42705098312484,0.992376807374411) circle (\grosseur);
\fill[racine] (2.56777171829279,0.889502215664309) circle (\grosseur);
\fill[racine] (2.43837481404927,1.40779635527655) circle (\grosseur);
\fill[racine] (1.71791931422417,0.417295498450807) circle (\grosseur);
\fill[racine] (2.54508497187474,1.19681767290924) circle (\grosseur);
\fill[racine] (1.56162518595073,1.40779635527655) circle (\grosseur);
\fill[racine] (1.52093707155071,1.34258203788318) circle (\grosseur);
\fill[racine] (2.16744427915491,0.358486905651613) circle (\grosseur);
\fill[racine] (2.07194665446632,1.63123494368546) circle (\grosseur);
\fill[racine] (1.92805334553368,1.63123494368546) circle (\grosseur);
\fill[racine] (2.40149162409079,0.518323457915138) circle (\grosseur);
\fill[racine] (2.12926463455446,1.62010435023380) circle (\grosseur);
\fill[racine] (2.04508497187474,0.330792269124804) circle (\grosseur);
\fill[racine] (1.83255572084509,0.358486905651613) circle (\grosseur);
\fill[racine] (1.43059868916688,1.05058125578611) circle (\grosseur);
\fill[racine] (2.47906292844929,1.34258203788318) circle (\grosseur);
\fill[racine] (2.31282031680318,1.53012616134654) circle (\grosseur);
\fill[racine] (2.21654236465910,1.58878972723441) circle (\grosseur);
\fill[racine] (1.46467783454561,0.750124755161338) circle (\grosseur);
\fill[racine] (1.56162518595073,1.40779635527655) circle (\grosseur);
\fill[racine] (1.41982127170454,1.00489903487844) circle (\grosseur);
\fill[racine] (2.53532216545439,0.750124755161338) circle (\grosseur);
\fill[racine] (1.87073536544554,1.62010435023380) circle (\grosseur);
\fill[racine] (1.68717968319682,1.53012616134654) circle (\grosseur);
\fill[racine] (2.49312970584268,0.652493977474185) circle (\grosseur);
\fill[racine] (1.95491502812526,0.330792269124804) circle (\grosseur);
\fill[racine] (1.45491502812526,1.19681767290924) circle (\grosseur);
\fill[racine] (1.59850837590921,0.518323457915138) circle (\grosseur);
\fill[racine] (1.78345763534090,1.58878972723441) circle (\grosseur);
\fill[racine] (1.56230589874905,0.565060654959379) circle (\grosseur);
\fill[racine] (2.28208068577583,0.417295498450807) circle (\grosseur);
\fill[racine] (1.95491502812526,0.330792269124804) circle (\grosseur);
\fill[racine] (2.43837481404927,1.40779635527655) circle (\grosseur);
\fill[racine] (1.69984026822590,0.420601387045538) circle (\grosseur);
\fill[racine] (1.50687029415732,0.652493977474185) circle (\grosseur);
\fill[racine] (2.56940131083312,1.05058125578611) circle (\grosseur);
\fill[racine] (2.58017872829546,1.00489903487844) circle (\grosseur);
\fill[racine] (2.43769410125095,0.565060654959379) circle (\grosseur);

\coordinate (ancre) at (-0.5,2.6);
\node[sommet,label=below left:{\huge $s_\alpha$}] (alpha) at (ancre) {};
\node[sommet,label=below right :{\huge $s_\beta$}] (beta) at ($(ancre)+(0.5,0)$) {} edge[thick] node[auto] {{\Large 5}} (alpha);
\node[sommet,label=above:{\huge $s_\gamma$}] (gamma) at
($(ancre)+(0.25,0.43)$) {} edge[thick] (alpha) edge[thick] (beta);

\end{tikzpicture}
}

&

\scalebox{0.66}{
\begin{tikzpicture}
	[scale=2,
	 q/.style={red,line join=round,thick},
	 racine/.style={blue},
	 racinesimple/.style={black},
	 racinedih/.style={blue},
	 sommet/.style={inner sep=2pt,circle,draw=black,fill=blue,thick,anchor=base},
	 rotate=0]


\node[anchor=south west,inner sep=0pt] at (0,0) {\includegraphics[width=8cm]{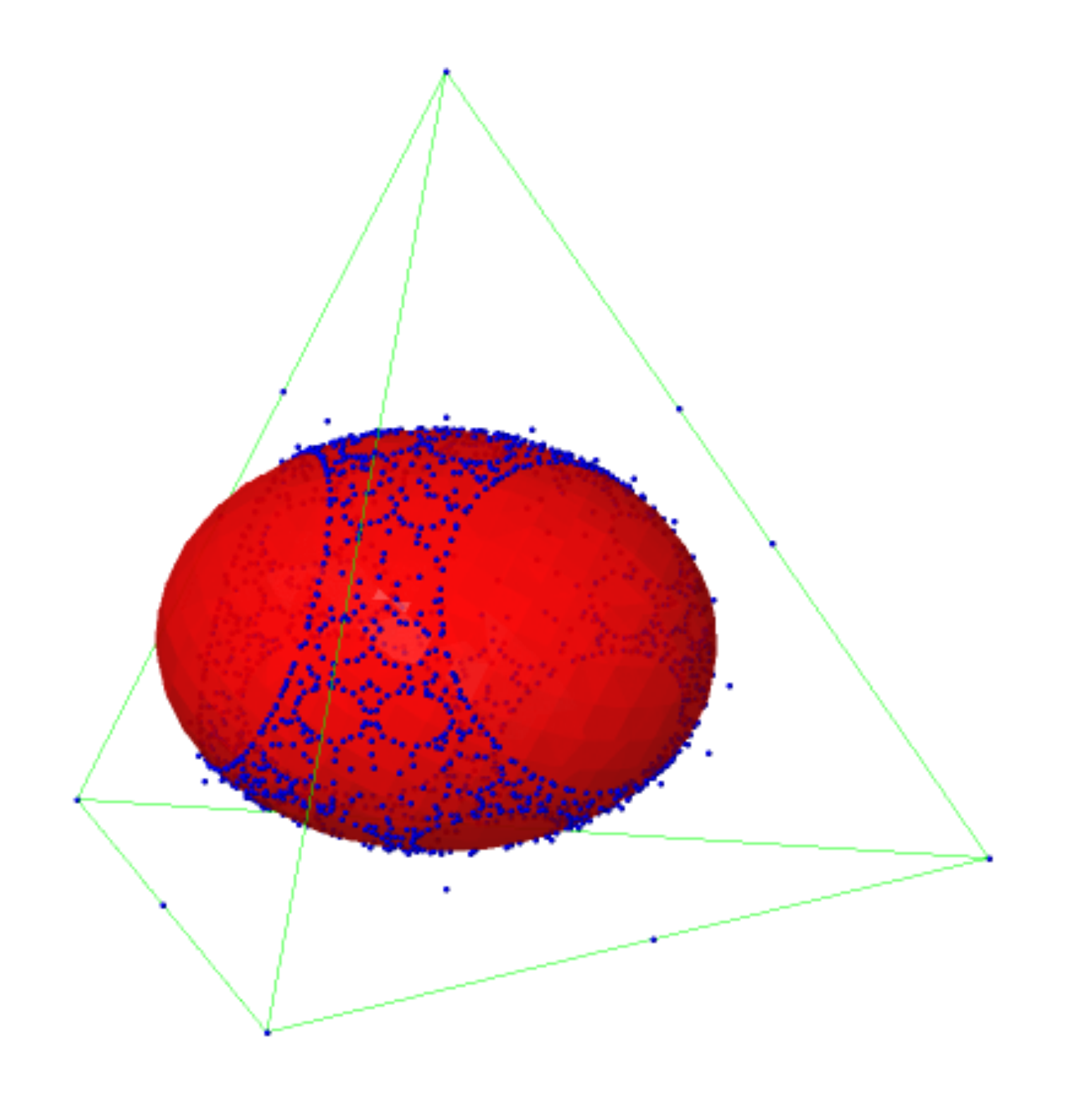}};

\coordinate (ancre) at (0,3.5);
\node[sommet,label=left:{\LARGE $s_\alpha$}] (alpha) at (ancre) {};
\node[sommet,label=right :{\LARGE $s_\beta$}] (beta) at ($(ancre)+(0.5,0)$) {};
\node[sommet,label=above:{\LARGE $s_\delta$}] (delta) at ($(ancre)+(0.22,0.43)$) {} edge[thick] node[auto,swap] {{\Large 4}} (alpha) edge[thick] node[auto] {{\Large 4}} (beta);
\node[sommet,label=below:{\LARGE $s_\gamma$}] (gamma) at ($(ancre)+(0.125,-0.215)$) {} edge[thick] node[auto,swap] {} (alpha) edge[thick] node[auto] {} (beta) edge[thick] node[midway, right] {{\Large 4}} (delta);

\end{tikzpicture}
}

\\
{\small (a) The first 100 normalized  roots, around the isotropic cone $Q$, for the rank~$3$ Coxeter group with the depicted graph.} & (b) {\small The first 1665 normalized  roots, around the isotropic cone $Q$, for the rank~$4$ Coxeter group with the depicted graph.}

\end{tabular}
\caption{Root systems for two infinite Coxeter groups computed via the computer algebra system Sage.}
\label{fig:sage}
\end{figure}

Let us explain what we see in these pictures. First, we fix a geometric action of $W$ on a finite dimensional real vector space~$V$,
which implies the data of a symmetric bilinear form~$B$, and a simple system $\Delta$, which is a basis for $V$ (the framework
we use is introduced in detail in \S\ref{sec:rep}). In~\S\ref{sec:limit}, we first show that the norm of an (injective) sequence of roots
diverges to infinity. So, in order to visualize ``limits'' of roots (actually the limits of their directions), we cut the picture by an affine
hyperplane. Define $V_1$ to be the affine hyperplane spanned by the points corresponding to the simple roots:
Figures~\ref{fig:sage}(a) and (b) live in~$V_1$ and the triangle (resp. tetrahedron) is the convex hull of the simple roots. The blue
dots are the intersection of $V_1$ with the rays spanned by the roots, and we call them \emph{normalized roots}. The red part
depicts the isotropic cone $Q=\{v\in V\,|\, B(v,v)=0\}$ of the quadratic form associated to $B$. We see on the pictures that the
normalized roots tend to converge to points on $Q$, and that the set of limit points has an interesting structure: it seems to be equal to $Q$ in Figure~\ref{fig:sage}(a), whereas in Figure~\ref{fig:sage}(b) it is similar to an Apollonian gasket.

Let $E$ be the set of accumulation points of the normalized roots, and
call \emph{limit roots} the points of $E$. After having explained our
framework in \S\ref{sec:rep}, we state our first result (Theorem
\ref{thm:inclusion} in \S\ref{subsec:limit}): the set $E$ is indeed
always contained in the isotropic cone $Q$.  M.~Dyer discovered
independently this property in his research on the imaginary cone of
Coxeter groups, see \cite{dyer:imaginary} and
Remark~\ref{rk:dyer}. However, we state this result and its proof in
an affine context, which is slightly different from M.~Dyer's
framework, and allows us to describe many examples and pictures in
rank $2$, $3$ and $4$. Through them, we see new geometric properties
emerging; in \S\ref{sec:action} and \S\ref{sec:dense} we describe two
of these properties of $E$ which we feel should motivate further works
on the subject:

\begin{enumerate}[(i)]
\item The geometric action of $W$ on $V$ induces an action on
  $E$. This action is simply given by the following process: for
  $\alpha\in\Delta$ and $x\in E$, the image~$s_\alpha\cdot x$ of $x$
  is the intersection point (other than $x$, if possible) of $Q$ with
  the line passing through the (normalized) root $\alpha$ and the
  point $x$ (see
  Prop.~\ref{prop:Estable}~and~\ref{prop:actiongeo}). Contrary to the
  usual action of $W$ on the roots, the action of $W$ on $E$ stays in
  the positive cone on $\Delta$; indeed $E$ lies in the convex hull
  $\conv(\Delta)$ of the simple roots.  Using this action, we give in
  \S\ref{subsec:fractal} some ways to understand the fractal
  phenomenon.

\item The set $E$ is the closure of the set of accumulation points obtained from the dihedral reflection subgroups of $W$ only.
Equivalently, $E$ is the closure of the set of all points you obtain by intersecting $Q$ with the lines in $V_1$ passing through two
normalized roots (see Theorem~\ref{thm:density}).
\end{enumerate}

A classical question when dealing with a property of Coxeter groups is how it passes down to subgroups (parabolic, or more
generally reflection subgroups). In the last section (\S\ref{sec:subgps}), we define the set of limit roots for a reflection subgroup,
and discuss how it compares to the set of limit roots for the whole group.

Along the text, we also present possible future directions and open problems.

\medskip

In a forthcoming paper (\cite{dhr}), the first and third authors, together with M.~Dyer, show that $E$ is the closure of the orbit of a finite set of accumulation points, and make some connections with the notions of root posets and of dominance order via the
imaginary cone for Coxeter groups (cf. \cite{dyer:imaginary,fu2}).

\medskip

\noindent\textbf{Figures.} The pictures were realized using the \TeX-package TikZ, and computed by dint of the computer
algebra system \emph{Sage} \cite{sage}.

%
%

\section{Geometric representations of a Coxeter group}
\label{sec:rep}

In this section, we recall some properties of Coxeter groups, the construction of their associated root systems, and we fix
notations. The theory of Coxeter groups is a rich one, and we recall here only what is necessary for the purpose of this paper.
For more details, see for instance \cite{bourbaki:cox,humphreys,kane,bjorner-brenti}, and the references therein.

We consider a Coxeter system $(W,S)$. Recall that $S\subseteq W$ is a set of generators for $W$, subject only to relations of
the form $(st)^{m_{s,t}}=1$, where $m_{s,t}\in \mathbb{N}^* \cup \{\oo\}$ is attached to each pair of generators $s,t\in S$, with
$m_{s,s}=1$ and $m_{s,t}\geq 2$ for $s\neq t$. We write $m_{s,t}=\oo$ if the product $st$ has infinite order. In the following,
we always suppose that $S$ is finite, and denote by $n=|S|$ the rank of $W$.

%

\subsection{The classical geometric representation of a Coxeter group} \label{subsec:cox}
~

Coxeter groups are modeled to be the abstract combinatorial counterpart of reflection groups, i.e., groups generated by
reflections.  It is well known that any finite Coxeter group can be represented geometrically as a (finite) reflection group. This
property still holds for infinite Coxeter groups, for some adapted definition of reflection that we first recall below.

\medskip

For $B$ a symmetric bilinear form on a real vector space $V$ (of finite dimension), and $\alpha \in V$ such that
$B(\alpha,\alpha)\neq 0$, we denote by $s_\alpha$ the following map:

\begin{equation}\label{eq:reflection}
s_{\alpha} (v) = v - 2\frac{B(\alpha,v)}{B(\alpha,\alpha)} \ \alpha , \quad \text{for any   } v\in V.
\end{equation}

We denote by $H_\alpha:=\{v\in V\,|\, B(\alpha,v)=0\}$ the orthogonal of the line $\mathbb{R}\alpha$ for the form $B$. Since
${B(\alpha,\alpha)\neq 0}$, we have $H_\alpha\oplus \mathbb{R} \alpha = V$. It is straightforward to check that~$s_\alpha$ fixes
$H_\alpha$, that~${s_\alpha(\alpha)=-\alpha}$, and~$s_{\alpha}$ also preserves the form~$B$, so it lies in the associated
orthogonal group~$\OO_B(V)$. We call~$s_\alpha$ the \emph{$B$-reflection associated to~$\alpha$} (or simply reflection
whenever~$B$ is clear). When $B$ is a scalar product, this is of course the usual formula for a Euclidean reflection.

\medskip

Let us now recall this \emph{classical geometric representation} (following \cite[\S5.3-5.4]{humphreys}). Consider a real vector
space $V$ of dimension $n$, with basis $\Delta=\{ \alpha_s ~|~ s\in S\}$. We define a symmetric bilinear form $B$ by:
\[
B(\alpha_s,\alpha_t)= \left\{
\begin{array}{ll}
-\cos \left( \frac{\pi}{m_{s,t}} \right) & \text{if }m_{s,t} < \oo, \\
-1 & \text{if } m_{s,t}=\oo.
\end{array}
\right.
\]

Then any element $s$ of $S$ acts on $V$ as the $B$-reflection
associated to $\alpha_s$ (as defined in
Equation~\eqref{eq:reflection}), i.e., $s (v) = v -
2B(\alpha_s,v)\ \alpha_s$ for $v \in V$. This induces
a \emph{faithful} action of $W$ on $V$, which preserves the form $B$;
thus we denote by the same letter an element of $W$ and its associated
element of $\OO_B(V)$.

%

\subsection{Root system and reflection subgroups of a Coxeter group} \label{subsec:reflsg}
~

The root system of $W$ is a way to encode the reflections of the
Coxeter group, i.e., the conjugates of elements of $S$, which are called simple reflections.
 The elements of~${\Delta=\{\alpha_s ~|~ s\in S\}}$ are
called \emph{simple roots} of~$W$, and the \emph{root system~$\Phi$}
of~$W$ is defined to be the orbit of $\Delta$ under the action of
$W$. By construction, any root~$\rho \in \Phi$ gives rise to the
reflection $s_{\rho}$ of $W$, which is conjugate to some~$s_{\alpha}
\in S$.

\medskip

A {\em reflection subgroup} of $W$ is a subgroup of $W$ generated by reflections; so it can be built from a subset of $\Phi$.
It turns out that any such reflection subgroup is again a Coxeter group, with some canonical generators
(\cite{deodhar:subgps,dyer:subgps}). So it is  natural to desire to apply results valid for $W$ to a reflection subgroup simply by
restriction. A major drawback of the classical geometric representation described above is that it is not ``functorial'' with
respect to the reflection subgroups: it can happen that the representation of some reflection subgroups~$W'$ of~$W$, induced
(by restriction) by the geometric representation of~$W$, is not the same as the geometric representation of~$W'$ as a Coxeter
group, see Example~\ref{ex:dihcex} below. Note that it does also happen that a Coxeter group of rank $n$ contains a reflection
subgroup of higher rank, as shown in Example~\ref{ex:HigherRank}.

\begin{ex}[Reflection subgroups of rank $2$] \label{ex:dihcex}
Let us consider the Coxeter group of rank $3$ with ${S=\{s_\alpha,s_\beta,s_\gamma\}}$ and $m_{s_\alpha,s_\beta}=5$,
$m_{s_\beta,s_\gamma}=m_{s_\alpha,s_\gamma}=3$ (whose Coxeter diagram is on Figure~\ref{fig:sage}(a)). Take the root
${\rho=s_\alpha s_\beta (\alpha)=s_\beta s_\alpha(\beta)}$, so that~$s_\rho$ corresponds to the longest element in the subgroup
$\left<s_\alpha,s_\beta\right>$: $s_\rho = s_\alpha s_\beta s_\alpha s_\beta s_\alpha=s_\beta s_\alpha s_\beta s_\alpha s_\beta$. We
compute $\rho= \frac{1+\sqrt{5}}{2} (\alpha + \beta)$. Consider the reflection subgroup $W'$ generated by $s_\gamma$ and $s_\rho$.
The product $s_\gamma s_\rho$ has infinite order, so $W'$ is an infinite dihedral group, with generators $s_\gamma$ and
$s_\rho$. But, if $B$ denotes the bilinear form associated to the Coxeter group $W$, we get: 
$B(\gamma, \rho) = - \frac{1+\sqrt{5}}{2}\neq -1$. So, the restriction to $W'$ of the geometric representation of $W$ does not
correspond to the classical geometric representation of $W'$ as an infinite dihedral group. In Example \ref{ex:rk3} we give a
geometric interpretation of this fact, which is visible in Figure~\ref{fig:sage}(a).
\end{ex}

%

\subsection{Coxeter groups from based root systems} \label{subsec:georep}
~

In order to solve the problem in Example~\ref{ex:dihcex}, we relax the requirements on the bilinear form~$B$ used to represent the group $W$: we allow
the values of some~$B(\alpha,\beta)$ to be any real numbers less than or equal to $-1$ (when the associated product of
reflections $s_{\alpha}s_{\beta}$ has infinite order). Actually, an even more general setting is better adapted here: the notion of
a \emph{based root system} (used for instance in \cite{howlett:coxeter,krammer:conjugacy,bonnafe-dyer}).

\begin{defi} \label{def:root}
Let $V$ be a real vector space, equipped with a bilinear form~$B$. Consider a finite subset $\Delta$ of $V$ such that

\begin{enumerate}[(i)]
\item $\Delta$ is positively independent\footnote{Geometrically, this means we require that $0$ does not lie in the convex
hull of the points of $\Delta$.}: if $\sum_{\alpha \in \Delta} \lambda_{\alpha} \alpha =0$ with all $\lambda_\alpha \geq 0$, then
all~${\lambda_\alpha=0}$;

\item for all $\alpha, \beta \in \Delta$, with $\alpha \neq \beta$, $\displaystyle{B(\alpha,\beta) \in (-\oo,-1] \cup \left\{-\cos\left(\tfrac{\pi}{k}\right), k\in \mathbb{Z}_{\geq 2} \right\} }$;

\item for all $\alpha \in \Delta$, $B(\alpha,\alpha)=1$.
\end{enumerate}

Such a set $\Delta$ is called a {\em simple system}. Denote by $S:=\{s_\alpha ~|~ \alpha \in \Delta\}$ the set of~$B$-reflections
associated to elements in $\Delta$ (see Equation~\eqref{eq:reflection}). Let $W$ be the subgroup of $\OO_B(V)$ generated by
$S$, and $\Phi$ be the orbit of $\Delta$ under the action of~$W$.

The pair $(\Phi,\Delta)$ is a \emph{based   root system in $(V,B)$}; its  {\em rank} is the cardinality of~$\Delta$, i.e., the cardinality
of $S$. We call the pair $(V,B)$  a {\em geometric $W$-module}\footnote{The triplet  $(V,\Delta,B)$ is sometimes called a
{\em Coxeter datum} in the literature, see for instance \cite{fu1,fu2}.}. We equivalently refer to {\em a geometric representation of
$W$} instead of a geometric~$W$-module.
\end{defi}

\begin{remark}\label{rk:positively}
~
\begin{itemize}
\item Condition~(ii) is natural to ensure that subrepresentations are
  again geometric representation in the sense of this new definition
  (we saw in Example~\ref{ex:dihcex} that this does not work for the
  usual definition).
\item In Condition~(i), the relaxation is more subtle, but also
  necessary if we want a nice functorial behaviour on the
  subrepresentations. For instance, for some Coxeter group $W$ there
  exists a reflection subgroup of rank (as a Coxeter group) strictly
  higher than that of $W$, see Example \ref{ex:HigherRank}.

\item Even if $\Delta$ is not anymore required to be a basis, the condition that it is positively independent is \emph{necessary} to
keep the usual properties of root systems, in particular the distinction between a set of \emph{positive} roots and a set of \emph{negative}
roots. Indeed, it is not difficult to prove that a set is positively independent if and only if its is included in an open half-space
supported by a linear hyperplane, see~\cite[ p.4 note (b)]{howlett:coxeter}.
\end{itemize}
\end{remark}

This generalization of root system enjoys the following expected properties (see for instance \cite{krammer:conjugacy,bonnafe-dyer}):
\begin{itemize}
 \item $(W,S)$ is a Coxeter system\footnote{This result is~\cite[Theorem 1.2]{vinberg}.}, where the order of $s_\alpha s_\beta$ is $k$ whenever $B(\alpha,\beta) = - \cos (\frac{\pi}{k})$, and~$\oo$ if $B(\alpha,\beta) \leq -1$.

 \item The convex set $\cone(\Delta)$ consisting of all positive linear combinations of elements of $\Delta$ allows us to define the set
 of {\em positive roots} $\pp:=\Phi \cap \cone(\DD)$, and then $\Phi =\pp \sqcup (-\pp)$ and $\mathbb
 R \rho\cap \Phi=\{\rho,-\rho\}$, for $\rho\in \Phi$.
\end{itemize}

The classical geometric representation (that we recalled in
\S\ref{subsec:cox}) is an example of such a geometric $W$-module. If
all $m_{s,t}$ (called the \emph{labels} of the group) are finite, then
the only possible representation (supposing that $\Delta$ is a basis)
is the classical one.  In particular, when the form $B$ is positive
definite, then~$\Phi$ is a \emph{finite} root system and contains no
more information than its associated finite Coxeter group.
\smallskip

We say that $(\Phi,\DD)$ is an \emph{affine based root system} when the form $B$ is positive semidefinite, but not definite.
Note that traditionally, the Coxeter group itself is said to be affine if the root system of its classical geometric representation
is affine.

\begin{ex}[Irreducible affine root systems] \label{ex:affine}
An infinite dihedral group $W$ has non-affine geometric representations as well as the classical affine representation.
 If $\Phi$ is an infinite root system of rank $2$, with
simple roots $\alpha$, $\beta$, then $B(\alpha,\beta) \leq -1$, and
$\Phi$ is affine if and only if $B(\alpha,\beta)=-1$ (i.e., when
$\Phi$ corresponds to the classical geometric representation of
$W$). We give a geometric description of these two cases in
Figure~\ref{fig:dih1}.  However, note that if $W$ is
irreducible\footnote{See \S\ref{subsec:examplesE} for the definition.}
of rank $\geq 3$, then $\Phi$ is affine if and only if $W$ is affine
(because there is no label $\oo$ in an irreducible affine Coxeter
graph of rank $\geq 3$).
\end{ex}

All the desired properties of the root system and of positive and negative roots still hold for a based root system. In particular,
the following statements, which are essential in the next sections, are still valid in this new framework.

\begin{prop}
  \label{prop:dihedral}
  Let $(\Phi,\Delta)$ be a based root system in $(V,B)$, with associated Coxeter system $(W,S)$.
  \begin{enumerate}[(i)]
  \item The set $\{B(\alpha,\rho)\,|\, \alpha\in\Delta,\ \rho\in \Phi^+\textrm{ and } |B(\alpha,\rho)|<1\}$ is finite.

  \item Denote by $Q$ the isotropic cone:
  \[Q
  := \{v \in V ~|~ q(v)=0 \}, \text{ where } q(v)=B(v,v).\] Let
  $\rho_1 \neq \rho_2$ be two roots in $\pp$. Denote by $W'$ the {\em
    dihedral reflection subgroup} of $W$ generated by the two
  reflections $s_{\rho_1}$ and $s_{\rho_2}$, and
  \[
  \Phi':= \{\rho \in \Phi ~|~ s_{\rho} \in W'\} \ .
  \]
  Then there exists $\Delta'\subseteq \Phi^+\cap\Phi'$ of cardinality $2$ such that   $(\Phi',\Delta')$ is a based root system of
  rank $2$, with associated Coxeter group $W'$. Moreover:
  \begin{enumerate}[(a)]
     \item $\Phi'$ is infinite if and only if the plane $\Span(\rho_1,\rho_2)$ intersects~$Q \setminus \{0\}$, if and only if $|B(\rho_1,\rho_2)|\geq 1$;

     \item $\Phi'$ is affine if and only if $\Span(\rho_1,\rho_2) \cap Q$ is a line, if and only if $B(\rho_1,\rho_2)=\pm1$;

    \item when $\Phi'$ is infinite, $\Delta'=\{\rho_1,\rho_2\}$ if and only if $B(\rho_1,\rho_2)\leq -1$. \label{item:last}
  \end{enumerate}
  \end{enumerate}
\end{prop}
\begin{proof}  The nontrivial parts of the proofs of these statements in the context of a based root system are word for word the
same as the proofs in the case of the root system of the classical geometric representation, proofs that the reader may find for
instance in \cite[\S4.5]{bjorner-brenti}. The last statement~(ii)(c)
is a consequence of \cite[Theorem 4.4]{dyer:subgps}, see also
\cite[Theorem 1.8~(ii)]{fu1}.
\end{proof}

%

\subsection{Other geometric representations} \label{subsec:georep2}
~

Let~$(W,S)$ be a Coxeter group. Fix a matrix $A=(a_{s,t})_{s,t\in S}$ such that

\begin{equation}
\left\{
\begin{array}{ll}
 a_{s,t}= -\cos\left(\frac{\pi}{m_{s,t}}\right) &\textrm{if } m_{s,t}<\oo,\\
   a_{s,t}\leq -1 &\textrm{if }m_{s,t}=\oo.
\end{array}
\right.
\label{equ:matrice}
\end{equation}

We associate to the matrix $A$ a {\em canonical geometric $W$-module} $(V_A,B_A)$ as follows.

\begin{itemize}
\item $V_A$ is a real vector space with basis  $\Delta_A = \{\alpha_s\,|\, s\in S\}$ and $B_A$ is the symmetric bilinear form
defined by $B_A(\alpha_s,\alpha_t)=a_{s,t}$ for $s,t\in S$.

\item Any element $s$ of $S$ acts on $V$ as the $B$-reflection associated to $\alpha_s$, i.e., $s (v) = v - 2B(\alpha_s,v)\ \alpha_s$
for $v \in V_A$.
\end{itemize}

Since $\Delta_A$ satisfies the requirement of Definition~\ref{def:root}, $W$ acts faithfully on $V_A$ as the subgroup of
~$\OO_{B_A}(V_A)$ spanned by the $B$-reflections associated to the $\alpha_s$. Moreover,~$(\Phi_A,\Delta_A)$ is a
based root system of $(V_A,B_A)$, where~$\Phi_A$ is the $W$-orbit of $\Delta_A$. Note that giving a matrix $A$, as we
did, is equivalent to fix the values in Conditions~(ii)~and~(iii) in Definition~\ref{def:root}.

\begin{ex}[Continuation of Example~\ref{ex:dihcex}]
In the case of Example~\ref{ex:dihcex}, the restriction of the classical geometric representation of $W$ to the reflection subgroup~$W'$
generated by $s_\gamma$ and~$s_\rho$ gives the geometric representation that is associated to the canonical geometric
$W'$-module given by the matrix

\[
A=\left(
\begin{array}{cc}
1&-\frac{1+\sqrt 5}{2}\\
-\frac{1+\sqrt 5}{2}&1
\end{array}
\right).
\]
\end{ex}
\begin{remark}
\label{rk:Deltabasis}
By construction, in the based root system $(\Phi_A,\Delta_A)$
associated to the canonical geometric $W$-module $(V_A,B_A)$ defined
above, the set $\Delta_A$ is a basis. This setting will actually be
the one used throughout the three next subsections:
\begin{center}\emph{In
  \S\S\ref{sec:limit}-\ref{sec:action}-\ref{sec:dense}, we assume that
  the set of simple roots is a basis for the vector
  space.}\end{center} Note that analogous results remain true in all
generality, but this assumption highly simplifies the constructions
and the statements, and allows us to get into the main subject more
quickly. However, for the sake of completeness, we added in
\S\ref{sec:subgps} a discussion of the case where~$\Delta$ is not a
basis.  Moreover, this turns out to be necessary in order to deal
thoroughly with the behaviour of restriction to reflection subgroups.
\end{remark}

%
%

\section{Limit points of normalized roots and isotropic cone}
\label{sec:limit}

Let $(\Phi,\Delta)$ be a based root system in $(V,B)$, with associated Coxeter group $W$ (as defined in Definition~\ref{def:root}).
We suppose that $\Delta$ is a basis for $V$; analogous results remain true in full generality (see \S\ref{sse:transverse} for details).

\medskip

When $W$ is finite,~$\Phi$ is also finite and the distribution of the
roots in the space~$V$ is well studied. However, when~$W$ is infinite,
the root system is infinite and we have, as far as we know, not many
tools to study the distribution of the roots over~$V$.  The
asymptotical behaviour of roots is one of them. This section deals
with a first step of this study: we show that the ``lengths'' of the
roots tend to infinity, and that the limit points of the
``directions'' of the roots are included in the isotropic cone of the
bilinear form associated to $\Phi$. In order to get a first grip of
what happens, we begin with some enlightening examples.

%

\subsection{Roots and normalized roots in rank $2$, $3$, $4$, and general setting} \label{subsec:ex}
~

Since $\Phi = \pp \sqcup (-\pp)$, it is enough to look at the positive roots, which are inside the simplicial cone $\cone(\DD)$.

\begin{ex}\label{ex:rk2} (Rank $2$: affine and non-affine representations of infinite dihedral groups){\bf .}
 Let $(\Phi,\Delta)$ be a based root system of rank $2$, as defined in
\S\ref{subsec:georep}. We get a Coxeter group~$W$ of rank ~$2$,
geometrically represented in a $2$-dimensional vector space~$V$
(together with a bilinear form~$B$), where $V$ is generated by two
simple roots~$\alpha,\beta$. Assume that $W$ is an infinite dihedral
group, so~${B(\alpha,\beta) \leq -1}$.

\noindent Suppose first that $B(\alpha, \beta) = -1$, i.e., that $\Phi$ is affine and with the classical geometric representation. Then
any positive root has the following form:
\[
\rho_n = (n+1)\alpha + n \beta,\textrm{ or }\rho'_n = n\alpha + (n+1)\beta,\textrm{ for } n\in \mathbb{N}.
\]

\noindent If we fix a (Euclidean) norm on $V$ (e.g., such that $\{\alpha,\beta\}$ is an orthonormal basis), then the norms of the
roots tend to infinity, but their directions tend to the line generated by $\alpha+\beta$ as depicted in Figure~\ref{fig:dih1}~(a). Note that this
line is precisely the isotropic cone of the bilinear form $B$, i.e., the set
\[
Q := \{v \in V ~|~ q(v)=0 \} \ \text{, where } q(v)=B(v,v)\ .
\]

\begin{figure}[!ht]
\centering
\captionsetup{width=0.9\textwidth}
\scalebox{0.9}{

\begin{tabular}{c@{\hspace{1cm}}c}

\begin{tikzpicture}
	[scale=1,
	 pointille/.style={densely dashed},
	 axe/.style={color=black, very thick},
	 rotate=45]

\coordinate (O) at (0,0);
\fill (O) circle (0.05);

\shade [shading=axis,top color=white!75!black,bottom color=white,shading angle=180] (O) -- (4,0) arc (0:90:4) -- (O);

\draw[axe,->] (O) -- (1,0) node[label=below right :{$\alpha=\rho_0$}] {};
\draw[axe,->] (O) -- (0,1) node[label=below left :{$\beta=\rho'_0$}] {};

\draw[->] (O) -- (2,1) node[label=right:{$\rho_1$}] {};
\draw[->] (O) -- (1,2) node[label=left:{$\rho'_1$}] {};

\draw[->] (O) -- (3,2) node[label=right:{$\rho_2$}] {};
\draw[->] (O) -- (2,3) node[label=left:{$\rho'_2$}] {};

\draw[->] (O) -- (4,3) node[label=right:{$\rho_3$}] {};
\draw[->] (O) -- (3,4) node[label=left:{$\rho'_3$}] {};

\draw[pointille,red] (-0.5,-0.5) -- (4,4) node[label=above right:{$Q$}] {};

\draw[pointille,color=green!50!black] (-1,2) -- (2,-1) node[label=right:{$V_1$}] {};

\end{tikzpicture}

&

\begin{tikzpicture}
	[scale=1,
	 pointille/.style={densely dashed},
	 axe/.style={color=black, very thick},
	 rotate=45]

\coordinate (O) at (0,0);
\fill (O) circle (0.05);

\def\arctandemi{40.97}

\shade [shading=axis,top color=white!75!black,bottom color=white,shading angle=180] (O) -- (4,0) arc (0:\arctandemi:4) -- (O);
\shade [shading=axis,top color=red!75!black,bottom color=white,shading angle=180] (O) -- (\arctandemi:4) arc (\arctandemi:90-\arctandemi:4) -- (O);
\shade [shading=axis,top color=white!75!black,bottom color=white,shading angle=180] (O) -- (90-\arctandemi:4) arc (90-\arctandemi:90:4) -- (O);

\draw[pointille,red] (-0.58,-0.5) -- (3.02,2.63) {};
\draw[pointille,red] (-0.5,-0.58) -- (2.63,3.02)  {};

\node[red] at (3,3) {$\ Q^-$};

\draw[axe,->] (O) -- (1,0) node[label=below right :{$\alpha=\rho_0$}] {};
\draw[axe,->] (O) -- (0,1) node[label=below left :{$\beta=\rho'_0$}] {};

\draw[->] (O) -- (2.02,1) node[label=right:{$\rho_1$}] {};
\draw[->] (O) -- (1,2.02) node[label=left:{$\rho'_1$}] {};

\draw[->] (O) -- (3.0804,2.02) node[label=right:{$\rho_2$}] {};
\draw[->] (O) -- (2.02,3.0804) node[label=left:{$\rho'_2$}] {};

\draw[->] (O) -- (4.202408,3.0804) node[label=right:{$\rho_3$}] {};
\draw[->] (O) -- (3.0804,4.202408) node[label=left:{$\rho'_3$}] {};

\draw[pointille,color=green!50!black] (-1,2) -- (2,-1) node[label=right:{$V_1$}] {};

\end{tikzpicture}

\\
\\
(a) $B(\alpha,\beta)=-1$ & (b) $B(\alpha,\beta)=-1.01<-1$

\end{tabular}

}
\caption{The isotropic cone $Q$ and the first positive roots of an infinite based root system of rank $2$. (a): in the (classical) affine
representation. (b): in a non-affine representation (the red part $Q^-$ denotes the set ${\{v\in V ~|~ q(v)<0 \}}$).}
\label{fig:dih1}
\end{figure}
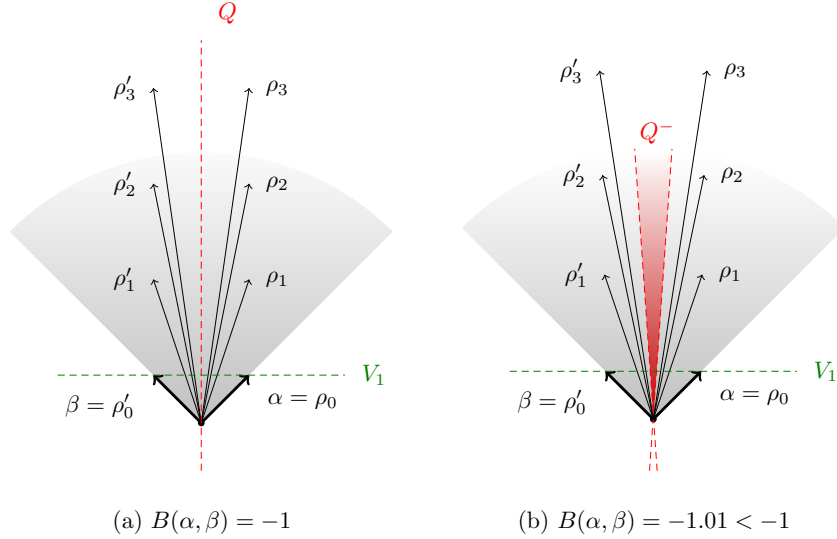

\noindent In a general geometric representation of $W$, $\Phi$ can be non-affine, i.e., $B(\alpha, \beta) =k$ with $k<-1$. Then
the isotropic cone $Q$ consists of two lines (generated by~{{$(-k\pm\sqrt{k^2-1})\alpha+\beta$}}). If we draw the roots, we note that, again,
their norms diverge to infinity and their directions tend to the two directions of the lines of $Q$; see Figure~\ref{fig:dih1}~(b), and
\cite[p.3]{howlett:coxeter} for a detailed computation.
\end{ex}

Let us go back to the general case of an infinite based root system of
rank~$n$. In the simple example of dihedral groups, we saw that the
roots themselves have no limit points; this phenomenon is actually
general, so we are rather interested in the asymptotical behaviour of
their directions. In order to talk properly about limits of
directions, we want to ``normalize'' the roots and construct ``unit
vectors'' representing each root.  One simple way to do so is to
intersect the line $\mathbb R \beta$ generated by a root $\beta$ in
$\Phi$ with the affine hyperplane~$V_1$ spanned by the simple roots
(seen as points), i.e., the affine hyperplane
\[
V_1 := \{v\in V ~|~ \sum_{\alpha \in \Delta} v_{\alpha}=1 \} \ ,
\]
where the $v_{\alpha}$'s are the coordinates of $v$ in the basis
$\Delta$ of simple roots. That way, we obtain what we call the set
of \emph{normalized roots}, denoted by $\hpp$:
\[
\hpp:=\bigcup_{\beta\in\Phi}  \mathbb R\beta \cap V_1.
\]
Let us describe this set more precisely.  Set
\[ \begin{array}{lclll}
V_0 & := & \{v\in V ~|~ |v|_1=0\}, & \text{and} \\
V_0^+ & := & \{v\in V ~|~ |v|_1>0\}, & \text{where} & |v|_1 := \sum_{\alpha \in \Delta} v_{\alpha}.\\
\end{array}
\]
Note that $|\cdot|_1$ is not a norm on $V$. Since all the positive roots are in the
half-space $V_0^+$, the entire root system $\Phi$ is contained in $V\setminus V_0$. So the following normalization map can
be applied to~$\Phi$:
\[
\begin{array}{ccl}
  V\setminus V_0 & \to & V_1 \\
  v & \mapsto & \widehat{v}:=\frac{v}{|v|_1}.
\end{array}
\]
For any subset $A$ of $V\setminus V_0$, write $\h{A}$ for the set $\{\h{a} ~|~ a \in A\}$. Because for $\rho\in \Phi$, $\mathbb
R \rho\cap \Phi=\{\rho,-\rho\}$, it is then obvious that $\Phi^+$ is in bijection with
\[
\hpp=\h{\pp}=\h{-\pp} = \{\h\rho\,|\,\rho\in\Phi^+\}.
\]

\begin{remark} \label{rem:projective}
Obviously, we could also have considered other affine hyperplanes to
``cut'' the rays of $\Phi$; it suffices that the chosen hyperplane be
``transverse to $\Phi^+$'', and this is discussed
in~\S\ref{sse:transverse}. We could also have considered the roots
abstractly, in the projective space $\mathbb{P} V$. The principal advantage
to consider an affine hyperplane explicitly, such as $V_1$, is to
visualize positive roots in an affine subspace of dimension~$n-1$,
inside an $n$-simplex (here $n=\dim V$).  Indeed, the simple roots are
in~$V_1$, so $\hpp$ lies in the convex hull $\conv(\Delta)$, which is
an $n$-simplex in $V_1$.  Note that as a convex polytope,
$\conv(\Delta)$ is closed and compact, which is practical when
studying sequences of roots.  From now on, in examples, we only draw
the normalized roots inside the $n$-simplex $\conv(\Delta)$.
\end{remark}

The aim of this work is to study the accumulation points of $\hpp$, i.e., the set of limit points of normalized roots. We first examine
its relation with the isotropic cone~$Q$.

\begin{ex}[Normalized roots in the dihedral case]
In the infinite dihedral case, the ``normalized'' version of Figure~\ref{fig:dih1} is Figure~\ref{fig:dih2}. Here $\hpp$ is contained in
the segment $[\alpha,\beta]$ and there are one or two limit points of normalized roots (depending on whether ${B(\alpha,\beta)=-1}$
or not), and the set of limit points is always equal to the intersection $Q \cap V_1 = \h{Q}$.

\begin{figure}[!ht]
\centering
\captionsetup{width=0.9\textwidth}
\scalebox{0.8}{
 \begin{tabular}{c@{\hspace{1cm}}c}

\begin{tikzpicture}
	[scale=1,
	 pointille/.style={densely dashed},
	 sommet/.style={inner sep=2pt,circle,draw=black,fill=blue,thick,anchor=base},
	 ]

\draw[pointille,color=green!50!black] (-2.5,0) -- (2.5,0) node[label=below right:{$V_1$}] {};

\coordinate (O) at (0,0);
\fill[red] (O) circle (0.05) node[label=below:{$\h{Q}$}] {};

\fill (2,0) circle (0.05) node[label=above :{$\alpha=\rho_1$}] {};
\fill (-2,0) circle (0.05) node[label=above :{$\beta=\rho'_1$}] {};

\fill (0.666,0) circle (0.05) node[label=above:{$\widehat\rho_2$}] {};
\fill (-0.666,0) circle (0.05) node[label=above :{$\widehat\rho'_2$}] {};

\fill (0.4,0) circle (0.05) ;
\fill (-0.4,0) circle (0.05) ;

\fill (2/7,0) circle (0.05) ;
\fill (-2/7,0) circle (0.05) ;

\node[label=above:{$\cdots$}] at (O) {};

\coordinate (ancre) at (-3,1);

\node[sommet,label=left:$s_\alpha$] (alpha) at (ancre) {};
\node[sommet,label=right:$s_\beta$] (beta) at ($(ancre)+(0.75,0)$) {} edge[thick] node[auto,swap] {$\infty$} (alpha);

\end{tikzpicture}

&

\begin{tikzpicture}
	[scale=1,
	 pointille/.style={densely dashed},
	 sommet/.style={inner sep=2pt,circle,draw=black,fill=blue,thick,anchor=base},
	 ]

\draw[pointille,color=green!50!black] (-2.5,0) -- (-0.1380,0);
\draw[pointille,color=green!50!black] (0.1380,0) -- (2.5,0) node[label=below right:{$V_1$}] {};

\coordinate (O) at (0,0);
\fill[red] (0.1380,0) circle (0.05);
\fill[red] (-0.1380,0) circle (0.05);

\draw[color=red] (-0.1380,0) -- (0.1380,0) {};

\fill (2,0) circle (0.05) node[label=above :{$\alpha=\rho_1$}] {};
\fill (-2,0) circle (0.05) node[label=above :{$\beta=\rho'_1$}] {};

\fill (0.675,0) circle (0.05) node[label=above:{$\widehat\rho_2$}] {};
\fill (-0.675,0) circle (0.05) node[label=above :{$\widehat\rho'_2$}] {};

\fill (0.4158,0) circle (0.05) ;
\fill (-0.4158,0) circle (0.05) ;

\fill (0.3081,0) circle (0.05) ;
\fill (-0.3081,0) circle (0.05) ;

\node[label=above:{$\cdots$}] at (O) {};
\draw[red] node[label=below:{\ $\h{Q^-}$}] at (O) {};

\coordinate (ancre) at (2.5,1);

\node[sommet,label=left:$s_\alpha$] (alpha) at (ancre) {};
\node[sommet,label=right:$s_\beta$] (beta) at ($(ancre)+(0.75,0)$) {} edge[thick] node[auto,swap] {$\oo(-1.01)$} (alpha);

\end{tikzpicture}

\\
\\
(a) $B(\alpha,\beta)=-1$ & (b) $B(\alpha,\beta)=-1.01<-1$

\end{tabular}

}
\caption{The normalized isotropic cone $\h{Q}$ and the first normalized roots of an infinite based root system of rank $2$. (a): in
the (classical) affine representation. (b): in a non-affine representation.}
\label{fig:dih2}
\end{figure}
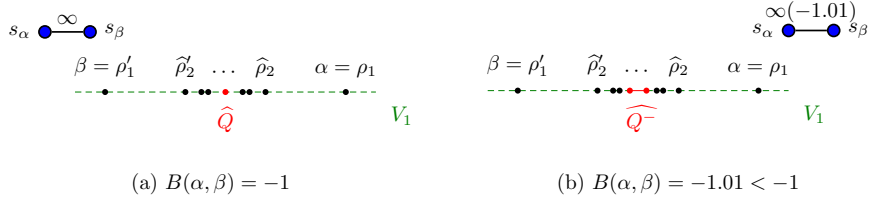
\end{ex}

\noindent \textbf{Notation.} The graph we draw to define a based root system is the same as the classical Coxeter graph, except
that, when the label of the edge $s_\alpha$---$s_\beta$ is $\oo$, we specify in parenthesis the value of $B(\alpha,\beta)$ if it is
not $-1$ (i.e., when we do not consider the classical representation).

\medskip

We give now some examples and pictures in rank $3$ and $4$.

\begin{ex}[Rank $3$]
\label{ex:rk3}
In Figures \ref{fig:sage}(a) (in the introduction) and~\ref{fig:EG2tilde} through~\ref{fig:Eoo}, we draw the normalized isotropic cone
$\h{Q}$ (in red), the $3$-simplex $\cone(\Delta)$ (in green), and the first normalized roots (in blue), for five different based root
systems of rank $3$. Note that the notion of depth used in the captions is a measure of the ``complexity'' of a root, which is defined
in \S\ref{subsec:limit}.

\begin{figure}[!ht]
\begin{minipage}[t]{.48\linewidth}
\centering
\captionsetup{width=\textwidth}
\scalebox{0.57}{



}
\caption{The normalized isotropic cone $\h{Q}$ and the first normalized roots (with depth $\leq 12$) for the based root system of
type $\widetilde{G_2}$ (affine).}
\label{fig:EG2tilde}
\end{minipage}
\hfill
\begin{minipage}[t]{.48\linewidth}
\centering
\captionsetup{width=\textwidth}
\scalebox{0.57}{

%

}
\caption{The normalized isotropic cone $\h{Q}$ and the first normalized roots (with depth $\leq 10$) for the based root system
with labels $2,3,7$.}
\label{fig:E237}
\end{minipage}
\end{figure}

\noindent The normalized roots seem again to tend quickly towards $\h{Q}$. In the affine cases, $\h{Q}$ contains
only one point, which is the intersection of the line $V^{\perp}$ (the radical of $B$) with $V_1$. In rank $3$, there are three
different types: $\widetilde{A_2}$, $\widetilde{B_2}$, and $\widetilde{G_2}$. The latter is drawn in Figure~\ref{fig:EG2tilde}.
Otherwise, $\h{Q}$ is always a conic (because the signature of $B$ is $(2,1)$),  and moreover it is always an ellipse in the
classical geometric representation (see~\S\ref{sse:transverse} for more details).

\medskip

\begin{figure}[!h]
\begin{minipage}[t]{.48\linewidth}
\centering
\captionsetup{width=\textwidth}
\scalebox{0.57}{



}
\caption{The normalized isotropic cone $\h{Q}$ and the first normalized roots (with depth  $\leq 10$) for the based root system
with labels $2,\oo(-1.1),\oo(-1.1)$.}
\label{fig:E-1.1}
\end{minipage}
\hfill
\begin{minipage}[t]{.48\linewidth}
\centering
\captionsetup{width=\textwidth}
\scalebox{0.57}{
%

%

}
\caption{The normalized isotropic cone $\h{Q}$ and the first normalized roots (with depth $\leq 8$) for the based root system
with labels $\oo, \oo(-1.5), 4$.}
\label{fig:Eoo}
\end{minipage}
\end{figure}

\noindent Some rank $2$ root subsystems appear in the pictures; they correspond to dihedral reflection subgroups. The
normalized roots corresponding to such a reflection subgroup, generated by two reflections $s_{\rho_1}$ and $s_{\rho_2}$,
lie in the line containing the normalized roots $\h{\rho_1}$ and $\h{\rho_2}$. Because of Proposition~\ref{prop:dihedral}~(ii),
the subgroup is infinite if and only if $\h{Q}$ intersects this line. In Figure~\ref{fig:sage}(a), for the group with labels $5,3,3$,
the line joining $\gamma$ and $\h{\rho}=\frac{\alpha + \beta}{2}$ intersects the ellipse in two points, as predicted by
Example~\ref{ex:dihcex}.

\noindent In general, the behaviour for standard parabolic dihedral subgroups is seen on the faces of the simplex, where three
situations can occur. The ellipse $\h{Q}$ can either cut an edge $[\alpha, \beta]$ in two points, or be tangent, or not intersect it,
depending on whether $B(\alpha,\beta) <-1$, $=-1$, or~$>-1$ respectively; see in particular Figures~\ref{fig:E-1.1} and~\ref{fig:Eoo}.
\end{ex}

\begin{remark} \label{rk:fractal}
When $\h{Q}$ is included in the simplex, it seems that the limit points of normalized roots cover the whole ellipse, whereas in
the other cases the behaviour is more complicated. We discuss this phenomenon in \S\ref{subsec:fractal}.
\end{remark}

\begin{ex}[Rank $4$]
\begin{figure}[!ht]
\begin{minipage}[b]{.48\linewidth}
\centering
\captionsetup{width=\textwidth}

\scalebox{0.75}{

\begin{tikzpicture}
	[scale=2,
	 q/.style={red,line join=round,thick},
	 racine/.style={blue},
	 racinesimple/.style={black},
	 racinedih/.style={blue},
	 sommet/.style={inner sep=2pt,circle,draw=black,fill=blue,thick,anchor=base},
	 rotate=0]


\node[anchor=south west,inner sep=0pt] at (0,0) {\includegraphics[width=8cm]{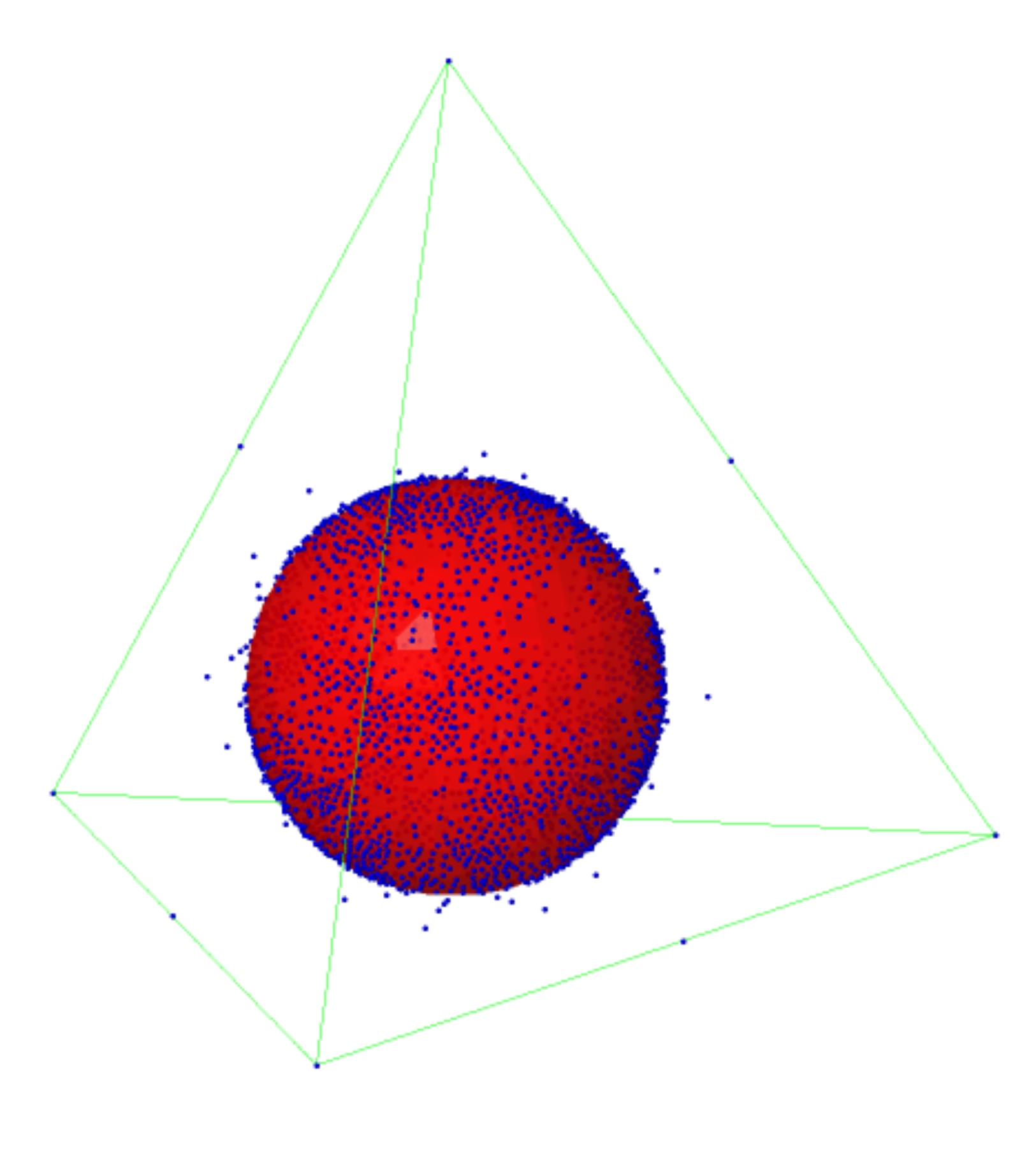}};

\coordinate (ancre) at (0,3.5);
\node[sommet,label=left:{\LARGE $s_\alpha$}] (alpha) at (ancre) {};
\node[sommet,label=right :{\LARGE $s_\beta$}] (beta) at ($(ancre)+(0.5,0)$) {} edge[thick] node[auto,swap] {} (alpha);
\node[sommet,label=above:{\LARGE $s_\delta$}] (delta) at ($(ancre)+(0.22,0.43)$) {} edge[thick] node[auto,swap] {} (alpha) edge[thick] node[auto] {} (beta);
\node[sommet,label=below:{\LARGE $s_\gamma$}] (gamma) at ($(ancre)+(0.125,-0.215)$) {} edge[thick] node[auto,swap] {} (alpha) edge[thick] node[auto] {} (beta) edge[thick] node[midway, right] {} (delta);

\end{tikzpicture}

}
\caption{The normalized isotropic cone $\h{Q}$ and the first normalized roots (with depth $\leq 8$) for the based root system with
diagram the complete graph with labels~$3$.}
\label{fig:E3d3}
\end{minipage}
\hfill
\begin{minipage}[b]{.48\linewidth}
\centering
\captionsetup{width=\textwidth}
\scalebox{0.75}{

\begin{tikzpicture}
	[scale=2,
	 q/.style={red,line join=round,thick},
	 racine/.style={blue},
	 racinesimple/.style={black},
	 racinedih/.style={blue},
	 sommet/.style={inner sep=2pt,circle,draw=black,fill=blue,thick,anchor=base},
	 rotate=0]


\node[anchor=south west,inner sep=0pt] at (0,0) {\includegraphics[width=8cm]{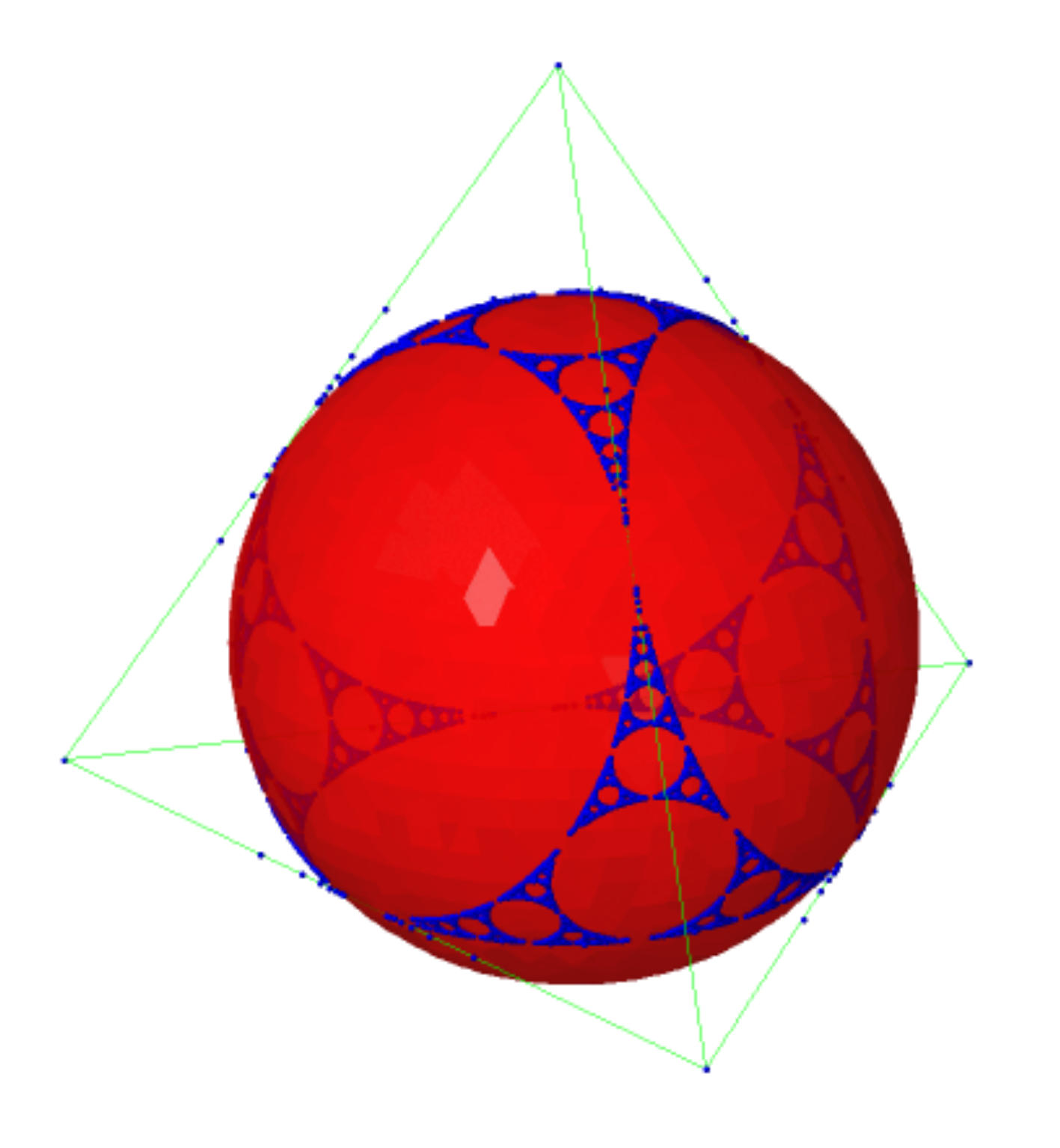}};

\coordinate (ancre) at (0,3.5);
\node[sommet,label=left:{\LARGE $s_\alpha$}] (alpha) at (ancre) {};
\node[sommet,label=right :{\LARGE $s_\beta$}] (beta) at ($(ancre)+(0.6,0.05)$) {} edge[thick] node[
near end, below] {{\footnotesize $\infty$}} (alpha);
\node[sommet,label=above:{\LARGE $s_\delta$}] (delta) at ($(ancre)+(0.27,0.5)$) {} edge[thick] node[auto,swap] {{\footnotesize $\infty$}} (alpha) edge[thick] node[auto] {{\footnotesize $\infty$}} (beta);
\node[sommet,label=below:{\LARGE $s_\gamma$}] (gamma) at ($(ancre)+(0.3,-0.25)$) {} edge[thick] node[auto] {{\footnotesize $\infty$}} (alpha) edge[thick] node[auto,swap] {{\footnotesize $\infty$}} (beta) edge[thick] node[auto
,right] {{\footnotesize $\infty$}} (delta);

\end{tikzpicture}

}
\caption{The normalized isotropic cone $\h{Q}$ and the first normalized roots (with depth $\leq 8$) for the based root system with
diagram the complete graph with labels~$\oo$.}
\label{fig:E3doo}
\end{minipage}
\end{figure}

Figures~\ref{fig:sage}(b) (in the introduction),
and~\ref{fig:E3d3}-\ref{fig:E3doo} illustrate some based root systems
of rank $4$, together with the tetrahedron $\conv(\Delta)$. Analogous
properties seem to be true: the limit points are in $\h{Q}$, and the
way how $\h{Q}$ cuts a facet depends on whether the associated
standard parabolic subgroup of rank $3$ is infinite non affine,
affine, or finite. Moreover, Remark~\ref{rk:fractal} still holds: in
Figure \ref{fig:E3d3} the limit points seem to cover the whole of
$\h{Q}$, whereas in Figures~\ref{fig:sage}(b) and~\ref{fig:E3doo},
some Apollonian gasket shapes appear. This fractal behaviour is
discussed in \S\ref{subsec:fractal}.
\end{ex}

%

\subsection{The limit points of normalized roots lie in the isotropic cone} \label{subsec:limit}
~

Recall that we denote by $q$ the quadratic form associated to $B$, and
by $Q$ the isotropic cone:
  \[Q := \{v \in V ~|~ q(v)=0 \}, \text{ where } q(v)=B(v,v).\]
The following theorem summarizes our first observations.

\begin{theo} \label{thm:inclusion}
Consider an injective sequence of positive roots $(\rho_n)_{n\in \mathbb{N}}$, and suppose that $(\h{\rho_n})$ converges to a limit
$\ell$. Then:
\begin{enumerate}[(i)]
\item the norm $||\rho_n||$ tends to infinity (for any norm on $V$);
\item the limit $\ell$ lies in $\h{Q}=Q\cap V_1$.
\end{enumerate}
In other words, the set $\h{Q}$ of accumulation points of normalized roots $\hpp$ is contained in the isotropic cone.
\end{theo}

\begin{remark} \label{rk:dyer}
M.~Dyer proved independently this property in the context of his work on imaginary cone \cite{dyer:imaginary}\footnote{M.~Dyer,
personal communication, September 2011.}, extending a study of V.~Kac (in the framework of Weyl groups of Lie algebras), who
states that the convex hull of the limit points correspond to the closure of the imaginary cone (see \cite[Lemma 5.8 and
Exercise~5.12]{kac}).
\end{remark}

Note that this theorem has the following consequence (which can of course be proved more directly using the fact that $W$ is
discrete in $\GL(V)$, see \cite{krammer:conjugacy} or \cite[Prop.~6.2]{humphreys}):

\begin{coro}
The set of roots of a Coxeter group is discrete.
\end{coro}

\begin{proof}
Suppose $\rho_n$ is an injective sequence converging to $\rho \in \pp$. Then $\h{\rho_n}$ converges to $\h{\rho}$, so by
Theorem~\ref{thm:inclusion}, $\h{\rho} \in Q$. Therefore $q(\h{\rho})=0$ which gives a contradiction since $q(\rho)=1$.
\end{proof}

The remainder of this subsection is devoted to the proof of Theorem
\ref{thm:inclusion}. We first need to recall the notion of depth of a
root. The depth of a positive root is a natural way to measure the
``complexity'' of this root in regard to the simple root it is
obtained from (see \cite[\S4]{Sa91} or  \cite[\S4.6]{bjorner-brenti} for details): for
$\rho \in \pp$,
\[
\dep(\rho)= 1 + \min\{k ~|~ \rho=s_{\alpha_1} s_{\alpha_2}\dots s_{\alpha_k}(\alpha_{k+1}), \text{ for } \alpha_1,\dot's, \alpha_k, \alpha_{k+1} \in \Delta\}.
\]
The depth is a very useful tool that allows inductive proof in infinite root systems: if~$\gamma$ is a root of depth $r\geq 2$,
then there is a root $\gamma'$ of depth $r-1$ and a simple root~${\alpha\in \Delta}$ such that $\gamma=s_\alpha(\gamma')$,
and moreover $B(\alpha,\gamma')<0$, see \cite[Lemma~4.6.2]{bjorner-brenti}.

Since $S$ is assumed to be finite, it follows that the number of positive roots of bounded depth is finite. Consider an injective sequence~$(\rho_n)_{n \in \mathbb{N}}$ of
positive roots, as in Theorem~\ref{thm:inclusion}. Then we obtain easily that $\dep(\rho_n)$ diverges to infinity as~$n\to \oo$. So, to prove
the first item of the theorem, it is sufficient to show that when the depth of a sequence of roots tends to infinity, so does the norm of the roots. This is done using the following lemma, which clarifies the relation between norm and depth.

\begin{lemma}
  \label{prop:dpnorm}
  Let $(\Phi,\DD)$ be a based root system, as defined in \S\ref{subsec:georep}. We take for the norm~$||.||$ the Euclidean
  norm for which $\Delta$ is an orthonormal basis for $V$.

  Then, with the above notations, we have the following properties:
  \begin{enumerate}[(i)]
  \item $\exists \kappa >0,\ \forall \alpha\in \Delta,\ \forall \rho \in \pp,\ B(\alpha,\rho)\neq 0 \ \Rightarrow \ |B(\alpha, \rho)|\geq \kappa$;

  \item $\exists \lambda >0, \ \forall \rho \in \pp,\ ||\rho||^2 \geq 1+\lambda (\dep (\rho) -1)$.
  \end{enumerate}
\end{lemma}

\begin{proof}
The first point is a direct consequence of
Proposition~\ref{prop:dihedral}~(i). Let us now prove, by
induction on $\dep(\rho)$, that
\[
\forall \rho \in \pp,\ ||\rho||^2 \geq     1+\lambda (\dep (\rho)-1),
\]
where $\lambda = 4\kappa ^2$ with $\kappa$ given by~(i). If
$\dep(\rho)=1$, $\rho \in \DD$ so $||\rho||=1 = 1+\lambda
(\dep(\rho)-1)$ by the choice of the norm $||\cdot||$. If
$\dep(\rho)=r \geq 2$, then we can write $\rho = s_{\alpha}(\rho')$,
with $\rho'\in \pp$ and $\alpha \in \DD$ such that $\dep(\rho')= r-1$
and $B(\alpha,\rho')<0$, by~\cite[Lemma~4.6.2]{bjorner-brenti}. We get
\begin{align*}
||\rho||^2 & = || \rho' - 2B(\alpha,\rho') \alpha ||^2 && \\
& = ||\rho'||^2 + 4B(\alpha,\rho')^2 - 4B(\alpha,\rho') \left< \alpha, \rho' \right>,&&
\end{align*}
\noindent where $\left< \cdot, \cdot\right>$ is the Euclidean product
of $||\cdot||$. But we know that $B(\alpha,\rho')<0$, and $\left<
\alpha, \rho' \right> \geq 0$ since $\rho'\in \cone (\Delta)$ and
$\Delta$ is an orthonormal basis for $||\cdot ||$. So we obtain by
induction hypothesis on $\rho'$ and by~(i):
\[
||\rho||^2 \geq ||\rho'||^2 + 4B(\alpha,\rho')^2 \geq 1+(r-2)\lambda + 4\kappa^2.
\]
Since $\lambda = 4\kappa^2$, we have $ ||\rho||^2   \geq 1 + (r-2)\lambda + \lambda  = 1 + (r-1)\lambda, $ which concludes the proof of~(ii).
\end{proof}

We can now finish the proof of Theorem~\ref{thm:inclusion}.

\begin{proof}[Proof of Theorem~\ref{thm:inclusion}]
As explained before Lemma~\ref{prop:dpnorm}, and by~(ii) of
this same lemma, the norm~$||\rho_n||$ of any injective sequence
$(\rho_n)_{n\in \mathbb{N}}$ in $\pp$ tends to infinity.

\noindent Recall from \S\ref{subsec:ex} that for $v \in V_0^+$, $\h{v}=\frac{v}{|v|_1}$, where $|v|_1=\sum_{\alpha\in \Delta}
v_{\alpha}$. If $\rho$ belongs to~$\pp$, the coordinates $\rho_{\alpha}$ are nonnegative, so $|\rho|_1$ is the $L_1$-norm
of $\rho$. In particular, by equivalence of the norms, $|\rho_n|_1$ tends to infinity as $||\rho_n||$ does.  We get
\[
q(\h{\rho_n})= q\left(\frac{\rho_n}{|\rho_n|_1}\right) = \frac{q(\rho_n)}{(|\rho_n|_1)^2} = \frac{1}{(|\rho_n|_1)^2} \ \xrightarrow[n\to \oo]{}\ 0.
\]
\noindent Supppose now that $\h{\rho_n}$ tends to a limit $\ell$. Then we obtain $q(\ell)=0$, i.e., $\ell \in Q$, which
concludes the proof of Theorem~\ref{thm:inclusion}.
\end{proof}

>From Theorem~\ref{thm:inclusion} and its proof, we also get these easy
consequences.

\begin{coro} The two following statements hold:
\begin{enumerate}[(i)]
\item for any $M\geq0$, the set $\{\rho \in \Phi ~|~ ||\rho|| \leq M\}$ is finite;

\item for any $\eps >0$, the set $\{\rho \in \Phi ~|~ \dist(\h{\rho},\h{Q}) \geq \eps \}$ is finite.
\end{enumerate}
\end{coro}

\begin{defi} \label{def:E}
Let $(\Phi, \Delta)$ be a based root system in $(V,B)$, and suppose that~$\Delta$ is a basis for $V$. We denote by
$E(\Phi)$ (or simply $E$ when there is no possible confusion) the set of accumulation points (or limit points) of
$\h{\Phi}$, i.e., the set consisting of all the possible limits of injective sequences of normalized roots.
\end{defi}

Sometimes, we refer to the points of $E$ as \emph{limit roots} of the root system $(\Phi, \Delta)$.

\begin{remark}
As stated in Remark~\ref{rem:projective}, we could have studied the roots in the projective space:
\[
\mathbb P\Phi:=\{\mathbb R\alpha\,|\, \alpha\in\Phi \}=\{\mathbb R\alpha\,|\, \alpha\in\Phi^+\}\subseteq \mathbb PV
\]
\noindent with the quotient topology. However, we choose not to for two reasons: the pictures at the base of this article
were obtained in an affine hyperplane and we wanted to explain precisely what we see; but also, since we do not use
projective geometry technics, it seemed natural to state our result in the simplest way possible. In this context, the
accumulation set $E(\Phi)$ of $\h{\Phi}$ is the accumulation set of $\mathbb P\Phi$, which is (since $\Phi$ is discrete) equal to:
$ \overline{\mathbb P \Phi}\setminus \mathbb P \Phi = \{ \mathbb Rx\,|\, x\in E(\Phi) \} $ (where $\overline{\mathbb P \Phi}$
denotes the topological closure of $\mathbb P \Phi$).
\end{remark}

As $\h{\Phi}$ is included in the simplex $\conv(\DD)$ (which is closed), Theorem \ref{thm:inclusion} implies \[ E(\Phi) \subseteq
Q \cap \conv(\DD)= \h{Q}\cap \cone(\DD)\  .\]

\noindent The reverse inclusion is not always true: we saw some examples of this fact in~\S\ref{subsec:ex}, for $\rk(W)\geq 4$, or even
for $\rk(W)=3$ whenever some $B(\alpha,\beta)<-1$. We address a more precise description of $E(\Phi)$ in~\S\ref{subsec:fractal}. 

Since $\Delta$ is finite, the convex set $\conv(\Delta)$ is compact. Moreover, if $\Phi$ is infinite, then $\h\Phi$ is infinite as well.
Therefore the inclusion $\h \Phi \subseteq \conv(\Delta)$ implies the following statement.

\begin{prop}\label{prop:nonempty} Let $(\Phi, \Delta)$ be a based root system in $(V,B)$, and suppose that~$\Delta$ is a basis for $V$. Then $E(\Phi)\not = \varnothing$ whenever  $\Phi$ is infinite.
\end{prop}

%

\subsection{Simple examples : rank $2$, reducible groups, and affine groups}  \label{subsec:examplesE}
~

When $(\Phi,\Delta)$ is an infinite based root system of rank $2$ (with $\Delta=\{\alpha,\beta\}$), we have already explained
that $E(\Phi)$ consists of one or two points (according to whether $\Phi$ is affine or not, i.e., whether $B(\alpha,\beta)=-1$ or
$<-1$): see Example \ref{ex:rk2}. In any case, we get $E(\Phi)=\h Q$.

For any based root system $(\Phi,\Delta)$, the limit roots coming from
a given rank~$2$ reflection subgroup can be observed inside
$E(\Phi)$. Take two distinct positive roots~$\rho_1$ and~$\rho_{2}$,
denote by $W'$ the dihedral reflection subgroup of $W$ generated by
the two reflections $s_{\rho_1}$ and $s_{\rho_2}$, and consider the
associated rank~$2$ root subsystem~${(\Phi',\Delta')}$ as in
Proposition~\ref{prop:dihedral}. Denote by $E(\Phi')$ the points of
$E$ that are limits of normalized roots of $\Phi'$, and by
$L(\h{\rho_{1}},\h{\rho_{2}})$ the line passing through $\h{\rho_{1}}$
and $\h{\rho_{2}}$. Then, by construction (and using
Proposition~\ref{prop:dihedral}), we have the following properties:

\begin{itemize}
\item $E(\Phi')=Q\cap L(\h{\rho_{1}},\h{\rho_{2}})= E(\Phi)\cap
  L(\h{\rho_{1}},\h{\rho_{2}})$;
\item the cardinality of $E(\Phi')$ is $0$, $1$ or $2$, depending on whether $|B(\rho_{1},\rho_{2})|<1$, $|B(\rho_{1},\rho_{2})|=1$
or $|B(\rho_{1},\rho_{2})|>1$.
\end{itemize}

Another case where the equality $E(\Phi)=\h Q$ holds is when $\Phi$ is
affine. To see that, we first need to discuss the reduction to
irreducible root systems. Recall that a Coxeter system $(W,S)$ is {\em
  irreducible} if there is no proper partition~$S=I\sqcup J$ such that
any element of $I$ commutes with any element of $J$, i.e., the Coxeter
graph is connected. Similarly, a based root system $(\Phi,\Delta)$ is
called {\em irreducible} if its associated Coxeter group is
irreducible, i.e., if there is no proper partition
${{\Delta=\Delta_I\sqcup \Delta_J}}$ such that $B(\alpha,\beta)=0$ for
all $\alpha\in \Delta_I$ and $\beta\in \Delta_J$.

Now, for any subset $\Delta_I$ of $\Delta$, we can associate a \emph{(standard) parabolic root subsystem} $(\Phi_I,\Delta_I)$
by setting:
\[
W_I:=\left< s_{\alpha} \, | \, \alpha \in \Delta_I\right> \quad \text{and} \quad \Phi_I:=W_I(\Delta_I).
\]
It is then natural to define the limit roots for this root subsystem
as the subset~$E(\Phi_I)$ of $E(\Phi)$ whose elements are limits of
sequences in $\h{\Phi_I}$. We postpone until \S\ref{sec:subgps} the
discussion on why this definition and its analogue for more general root
subsystems (not necessarily standard parabolic) make sense. Now we explain
 in the following easy proposition why it is possible to limit our study to irreducible root systems.

\begin{prop}
\label{prop:irred}
Assume that $(\Phi,\Delta)$ is reducible, and consider proper subsets $\Delta_{I},\Delta_J\subset \Delta$ such that
$\Delta=\Delta_I\sqcup \Delta_J$, with $B(\alpha,\beta)=0$ for all $\alpha\in \Delta_I$ and~$\beta\in \Delta_J$. Then
$E(\Phi)=E(\Phi_I)\sqcup E(\Phi_J)$.
\end{prop}

\begin{proof}
We have $\widehat{\Phi}=\widehat{\Phi_I}\cup\widehat{\Phi_J}$. Since
$\widehat{\Phi_I}\subseteq \conv(\Delta_I)$ (resp.
$\widehat{\Phi_J}\subseteq \conv(\Delta_J)$) and since
$\conv(\Delta_I)$ and $\conv(\Delta_J)$ are disjoint compact sets, a
converging sequence of elements in $\widehat{\Phi}$ eventually lives
either in $\conv(\Delta_I)$ or in $\conv(\Delta_J)$, and so converges
to a limit point either in $E(\Phi_I)$ or in $E(\Phi_J)$.
\end{proof}

\begin{coro} \label{rem:affine}
If $(\Phi,\Delta)$ is an affine based root system, then $E(\Phi)$ is finite. Moreover, if $(\Phi,\Delta)$ is affine and irreducible, then
$E(\Phi)$ is a singleton (and is equal to $\h Q$).
\end{coro}

\begin{proof}
By definition, if $(\Phi,\Delta)$ is an irreducible and affine based root system, of rank $n$, then the signature of $B$ is $(n-1,0)$.
So $\h{Q}$ is a point in this case, and since~$\Phi$ is infinite, $E\neq \varnothing$ by Proposition~\ref{prop:nonempty}. Since $E\subseteq \h{Q}$, we get that
$E=\h{Q}$ and contains a single element. The proof follows then from Proposition~\ref{prop:irred}.
\end{proof}

%
%

\section{Geometric action of the Coxeter group $W$ on the limit roots} \label{sec:action}

In this section, we address some natural questions concerning the
pictures we obtained: can $W$ act on the set $E$ of limit roots? Where
does the ``fractal phenomenon'' come from? As in the previous section,
we fix an infinite based root system~${(\Phi,\DD)}$ in $(V,B)$ with
associated Coxeter system $(W,S)$, and we assume that $\Delta$ is a
basis for $V$.  We recall that $V_1$ denotes the affine hyperplane
containing the simple roots (seen as points), with direction $V_0$,
and $E=E(\Phi)$ the set of limit roots, i.e., the accumulation set of
the normalized roots.

%

\subsection{Geometric action} \label{subsec:action}
~

It is clear that the geometric action of the group $W$ on $V$ does not restrict to an action on $V_1$. However, using the
normalization map, it induces a natural action on a part of $V_1$.  We consider the set

\begin{align*}
D&:= \bigcap_{w\in W} w(V\setminus V_0) \cap V_1\\
&\phantom{:}= V_1\setminus\bigcup_{w\in W} w(V_0) .
\end{align*}

We define an action of $W$ on $D$ as follows: for $x\in D$ and $w\in W$, set
\[
w\cdot x := \h{w(x)}\ .
\]

Since any element of $D$ stays in the complement of $V_0$ after action of any $w\in W$,
it is straightforward to check that this is a well-defined action of
$W$, and that any $w\in W$ acts continuously on $D$. The set $D$ is
the complement in $V_1$ of an infinite union of hyperplanes, and is
the maximal subset of $V_1$ on which $W$ acts naturally. The following
statement is the main motivation to consider this action.

\begin{prop}\label{prop:Estable}
~
\begin{enumerate}[(i)]
\item $\hpp$ and $E$ are contained in $D$;
\item $\hpp$ and $E$ are stable by the action of $W$; moreover
  $\h\Phi=W\cdot \Delta$;
\item the topological closure $\hpp\sqcup E$ of $\hpp$ is stable by
  the action of $W$.
\end{enumerate}
\end{prop}

\begin{proof} First, we show that $\hpp$ is contained in $D$ and that $\hpp$ is stable by $W$. Recall that  $\Phi$ lies in $V\setminus V_0$ and is stable under the action of $W$ on $V$. Let $\h\beta
\in \hpp$ and~${w\in W}$.  Since $\mathbb R\beta = \mathbb R \h\beta$,
we have $\mathbb R w(\h\beta)=w(\mathbb R \h\beta) = w(\mathbb R
\beta)=\mathbb R w(\beta)$. Therefore, since $w(\beta)$ is a root, the
point $w(\h\beta)$ lies in $V\setminus V_0$ and $w\cdot \h\beta =
\widehat{w(\h\beta)}= \h{w(\beta)}$ is a point of $\hpp$. It remains
to prove that:
\[
\forall w \in W, \forall x\in E,\ w(x) \in V\setminus V_0 \text{ and } w \cdot x =  \h{w(x)} \in E  \ .
\]
Take $x\in E$, and $w \in W$. Consider an injective sequence~$(\rho_n)_{n\in \mathbb{N}}$ in $\pp$, such that~$(\h{\rho_n})$
converges to $x$. Then, for all $n$,  $w(\h{\rho_n})$ is in $V\setminus V_0$, as well as its limit~$w(x)$. Since $w$ and $\h{\cdot}$
are continuous, we have:
\[
w \cdot \rho_n = \h{w (\rho_n)} = \h{w (\h{\rho_n})} \xrightarrow[n\to \oo]{} \h{w (x)} = w\cdot x\ .
\]
So $w \cdot x$ is a limit point of the set $\hpp$, i.e., lies in
$E$. The fact that $\h\Phi=W\cdot \Delta$ follows directly from the
the definition of $\Phi$, $\Phi=W(\Delta)$. So~(i) and~(ii) are
proved, and~(iii) is a trivial consequence of them.
\end{proof}

It is possible to restrict the action of $W$ on $D$ to the simplest connected component of $D$, which is the convex set:
\[
D^+:= \bigcap_{w\in W} w(V_0^+) \cap V_1\subseteq V_0^+\cap V_1.
\]
This convex set is still stable by the action of $W$ and does not contain any element of $\h\Phi$ (since $s_\alpha(\alpha)=-\alpha$).
However $D^+$ is of interest to study the action of $W$ on~$E$ (see for instance the proof of Proposition~\ref{prop:visible}):

\begin{prop}  \label{prop:D+}
The set $E$ is contained in $D^+$. In particular, $W(E)$ and $W\cdot E$ are subsets of $V_0^+$, that is, $|w(x)|_1>0$, for all
$x\in E$ and $w\in W$.
\end{prop}

\begin{proof}
We need to prove that for any $w \in W$ and $x\in E$, the image $w(x)$
stays in~$V_0^+$. Using a direct induction on the length of $w$---that
is, the minimal length of the word $w$ in the alphabet~$S$---this is a
consequence of the following fact:
\begin{equation*}
\forall \alpha \in \DD, \forall x\in E,  s_{\alpha}(x) \in V_0^+.
\end{equation*}
Indeed we know from Proposition \ref{prop:Estable} that $E$ is stable
by the action of $W$. Let us prove this fact. Take $x\in E$, and
$\alpha \in \DD$. Consider an injective sequence~$(\rho_n)_{n\in\mathbb{N}}$
in~$\pp$, such that $(\h{\rho_n})$ converges to $x$. We can suppose
that $(\rho_n)$ does not contain~$\alpha$. Then, for all $n$,
$s_{\alpha}(\rho_n)$ is in~$\pp$, and in particular in~$V_0^+$, so
$s_{\alpha}\cdot \h{\rho_n}$ is also in $V_0^+$, as well as its limit
$s_{\alpha}\cdot x$.
\end{proof}

\begin{remark}
Although  we do not know any better description of $D^+$ than the definition given here,  M.~Dyer suggested that if we replace
$V_0^+$ with $\cone(\DD)$ in the definition of $D^+$, then the convex set we obtain may be equal to $\conv(E)$. We do not
know a proof of this statement or even if $D^+$ is  equal to this new convex set.
\end{remark}

\begin{remark}
Obviously, this action is not faithful when $\Phi$ is an affine root system (since~$E$ is  finite, see Corollary~\ref{rem:affine}). We
do not know whether this action is faithful for  infinite irreducible non affine root systems of rank $\geq 3$.
\end{remark}

 The aim behind studying the set $E \subseteq V_1$ is to be able to represent ``limit points of roots'' in an affine space. Now we can
also study an action of $W$ on these limit points, and it turns out that the action of the reflections of~$W$ on $E$ (and more generally
on $Q\cap D$) is geometric in essence:

\begin{prop}[A geometric description] \label{prop:actiongeo}
  ~
  \begin{enumerate}[(i)]
  \item  Let $\rho \in \Phi$, and $x\in D\cap Q$. Denote by $L(\h{\rho},x)$ the line containing~$\h{\rho}$ and~$x$. Then:
    \begin{enumerate}[(a)]
    \item if $B(\rho,x)=0$, then $L(\h{\rho},x)$ intersects $Q$ only in $x$, and $s_{\rho} \cdot x = x$;
    \item if $B(\rho,x)\neq 0$, then $L(\h{\rho},x)$ intersects $Q$ in the two distinct points $x$ and~$s_{\rho}\cdot x$.
    \end{enumerate}
  \item Let $\rho_1 \neq \rho_2\in\Phi$, $x\in L(\h\rho_1,\h\rho_2)\cap Q$ and $w\in W$. Then  $w\cdot x\in L(w\cdot\h\rho_1,w\cdot\h\rho_2)\cap Q$.
  \end{enumerate}
\end{prop}

To visualize this proposition, we refer to Figure~\ref{fig:Eoo} in
which we draw the images of the two yellow points $x$ and $y$ on the
edge $[\beta,\gamma]$ of $\conv(\Delta)$ under the action
of~$s_\alpha$ and of~$s_\beta s_\alpha$.

\begin{proof}
(i) Let $\rho \in \Phi$, and $x\in D\cap Q$. First, we know that $s_\rho (x)$ and $s_{\rho} \cdot x$ belong to $Q$. Note also that
$s_{\rho} \cdot x \in \Span (\rho,x) \cap V_1=L(\h{\rho},x)$, so $x$ and $s_{\rho}\cdot x$ are always intersection points of
$L(\h{\rho},x)$ with $Q$. Since $s_{\rho}$ is a reflection and $x\in V_1$, we have
\[
|L(\h{\rho},x)\cap Q |= 1 \ \Rightarrow \ s_{\rho}\cdot x =x \ \Leftrightarrow \ s_{\rho}(x)=x \ \Leftrightarrow B(\rho,x)=0\ .
\]
Consider an element $u_\lambda=\lambda \h{\rho} + (1-\lambda) x$ in $L(\h{\rho},x)\cap Q$.
Then
\begin{eqnarray*}
0  =  q(u_\lambda) &=& \lambda^2 q(\h{\rho}) + 2\lambda(1-\lambda)B(\h{\rho},x) + (1-\lambda)^2 q(x)\\
&=& \lambda^2 q(\h{\rho}) + 2\lambda(1-\lambda)B(\h{\rho},x).
\end{eqnarray*}
Hence $L(\h\rho,x)$ intersects $Q$ in at most two points. In the case that $B(\rho,x)\not =0$, the two points $x$, $s_\rho\cdot x$ are distinct, and thus the third sentence of the present proof yields that $L(\h{\rho},x)\cap Q=\{x,s_\rho\cdot x\}$. In the case that
 $B(\rho,x)=0$, then $0= q(u_\lambda) = \lambda^2q(\h \rho)$, so  $u_{\lambda}=x$ and $L(\h{\rho},x)\cap Q=\{x\}$.

\smallskip

\noindent (ii) Let $\rho_1 \neq \rho_2\in\Phi$. For any $x\in E\subseteq Q$ and  $w\in W$,  since $q(x) = B(x,x) = B(w(x),w(x)) = q(w(x))$, it follows that both $w(x)$ and $w\cdot x$ are in $Q$. Therefore, we only need to check that if $x \in L(\h\rho_1,\h\rho_2)$
then $w\cdot x\in L(w\cdot\h\rho_1,w\cdot\h\rho_2)$.

Let  $x\in L(\h\rho_1,\h\rho_2)\cap Q$. Note that $x\in E$ as explained in~\S\ref{subsec:examplesE}.  Let $w\in W$ and set $\rho_1':=w(\rho_1)$ and $\rho_2':=w(\rho_2)$. Then $w\cdot\rho_1=\h{\rho_1'}$ and
$w\cdot\rho_2=\h{\rho_2'}$.  Since $x\in L(\h\rho_1,\h\rho_2)\subseteq \Span (\rho_1,\rho_2 )$,
it follows that $w(x)$ is a point in the plane $P$ spanned
by~$\rho_1'$ and~$\rho_2'$. Since $P\cap V_1 =
L(\h{\rho_1'},\h{\rho_2'})$, and $w\cdot x\in P\cap V_1$, it follows
that $w\cdot x\in  L(\h{\rho_1'},\h{\rho_2'})\cap Q$ as required.
\end{proof}

As we can see on Figure~\ref{fig:Eoo}, when $L(\h{\rho},x)$ intersects $Q$  in two points, one of them is closest to $\h\rho$ than the
other; the closest point to $\h\rho$ is then {\em visible from $\h\rho$}. Let us describe this phenomenon more precisely.

\begin{defi} We say that $x\in \h Q$ is {\em visible from $v\in V_1$} if the segment $[v,x]$ intersects $\h Q$ in only one point,
which has to be $x$; in other words, $x$ is visible from~$v$ if and only if $[v,x]\cap Q=\{x\}$.
\end{defi}

The following proposition sums up what we observe in Figure~\ref{fig:Eoo}.

\begin{prop}
\label{prop:visible}
Let $x\in E$ and $\rho\in \Phi$. Then,
\begin{enumerate}[(i)]
\item $x$ is visible from $\h\rho$ if and only if $B(\h\rho,x)\geq 0$;
\item if $\rho\in\Phi^+$,  $x$ is visible from $\h\rho$ if and only $B(\rho,x)\geq 0$;
\item if $\rho\in\Phi^+$ and $w\in W$ such that $w(\rho)\in\Phi^+$,
  then $x$ is visible from $\h\rho$ if and only $w\cdot x$ is visible
  from $w\cdot\h\rho$.
\end{enumerate}
\end{prop}
\begin{proof}
If $B(\rho,x)=0$, then by Proposition~\ref{prop:actiongeo},
$L(\h\rho,x)\cap Q=\{x\}$, so both statements are obviously true.
Assume now that $B(\rho, x)\not = 0$, which is equivalent by
Proposition~\ref{prop:actiongeo} to $|L(\h\rho,x)\cap Q|=2$. First
suppose that $\rho\in\Phi^+$. So the linear forms~${B(\rho, \cdot)}$ and
$B(\h\rho,\cdot)$ are of the same sign. In this case, the proof
follows from the equality $s_{\rho}(x)=x - 2 B(\rho,x) \rho$, which
implies that when $B(\rho,x)<0$, $s_{\rho}(x)$ is in the interior of
$\cone(x,\rho)$, and when $B(\rho,x)>0$, $x$ is in the interior of
$\cone (s_{\rho}(x),\rho)$. Now, if $\rho\in\Phi^-$, apply the above
case with $-\rho\in\Phi^+$, which proves~(i) and~(ii).

\noindent For~(iii), we use Proposition~\ref{prop:D+}: $E\subseteq
D^+$, therefore $w(x)\in V_0^+$. So~${|w(x)|_1>0}$, which means that
$B(w(x),\cdot)$ and $B(w\cdot x,\cdot)$ have the same sign. Since we
have ${w(\rho)\in \Phi^+\subseteq V_0^+}$, $B(\cdot,w(\rho))$ and
$B(\cdot, w\cdot \h\rho)$ have also the same sign. We conclude
using~(i) and~(ii).
\end{proof}

We end this discussion by  completing the example of infinite dihedral groups.

\begin{ex}[The case of infinite dihedral groups]  \label{ex:rk2order}
Let $(\Phi,\Delta)$ be of rank~$2$ as in Example~\ref{ex:rk2}   (illustrated in Figure~\ref{fig:dih1} and Figure~\ref{fig:dih2}). Recall
that $\Delta=\{\alpha,\beta\}$ with $B(\alpha,\beta)\leq -1$ and that $V_1$ is the line $L(\alpha,\beta)$. Let $x\in \h Q$ such that $B(x,\alpha)\geq 0$, so $x$ is
visible from $\alpha$. Since $s_\beta\cdot \alpha\in V_1 = L(\alpha,\beta)$ and $s_\beta\cdot \beta=\beta\in V_1 = L(\alpha,\beta)$, it follows from Proposition~\ref{prop:actiongeo}~(ii) that $s_\beta\cdot x\in L(\alpha,\beta)\cap Q$, and similarly that $s_\alpha\cdot x\in L(\alpha,\beta)\cap Q$. Now since $L(\alpha,\beta)=L(x,\alpha)=L(x,\beta)$, it follows from Proposition~\ref{prop:actiongeo}~(i) that either $x=s_\beta\cdot x$ and $x=s_\alpha\cdot x$ or $L(\alpha,\beta)\cap Q = \{x,s_\beta\cdot x\}= \{x,s_\alpha\cdot x\}$ has cardinality $2$. Thus in both case,  $s_\beta\cdot x=s_\alpha\cdot x$, and this point is
visible from  $\beta$. So the segment $[\alpha,x]$ intersects $\h Q$ only in $x$.

\noindent We endow the line $L(\alpha,\beta)=V_1$  with the total order on $\mathbb R$ by orienting it from $\alpha$ to $\beta$. So from the above discussion  we get that
\[
\alpha<x\leq s_\alpha\cdot x<\beta.
\]
We are particularly interested in the following partition of the segment $[\alpha,\beta]$:
\[
[\alpha,\beta]=[\alpha,x)\sqcup[x,s_\alpha\cdot x]\sqcup(s_\alpha\cdot x,\beta].
\]
 It is straightforward computations\footnote{See for instance \cite[p.5 Eq. (1.2)]{fu1} for details.} to prove the following  nice geometric
 description of  the action of $s_\alpha$ and $s_\beta$ on $\h\Phi\sqcup E$:
\[
\alpha<s_\alpha\cdot \beta<s_\alpha s_\beta \cdot \alpha< s_\alpha s_\beta s_\alpha\cdot \beta<\dots<(s_\alpha s_\beta)^n\cdot \alpha< (s_\alpha s_\beta)^n s_\alpha\cdot \beta <\dots <x
\]
and
\[
s_\beta\cdot x = s_\alpha\cdot x< \dots<s_\beta (s_\alpha s_\beta)^n\cdot \alpha< (s_\beta s_\alpha)^n \cdot \beta<\dots <s_\beta s_\alpha \cdot \beta< s_\beta \cdot \alpha< \beta.
\]
Moreover, we note the following two facts:
\begin{itemize}
\item the bilinear form $B$ is positive on {{$\left[\alpha,x\right[
        \times \left[\alpha,x\right[$}} and $\left]s_\alpha\cdot
      x,\beta\right] \times \left]s_\alpha\cdot x,\beta\right]$, and
  negative on $\left[\alpha,x\right[ \times \left]s_\alpha\cdot
    x,\beta\right]$;

\item the depth ($\dep$) is increasing on $\left[\alpha,x\right[ \cap
    \h\Phi$ and decreasing on $\left]s_\alpha\cdot x,\beta\right]\cap
  \h\Phi$ (where we set $\dep(\h\rho):=\dep(\rho)$ for $\rho \in
  \pp$).
\end{itemize}
\end{ex}

%

\subsection{Fractal description of $E$}
\label{subsec:fractal}
~

We know that $E$ is contained in $\h{Q}$
(Theorem~\ref{thm:inclusion}), but we could obtain a more precise
inclusion, using the action of $W$. First of all, $E$ is also included
in $\conv(\DD)$. Now consider the example illustrated in
Figure~\ref{fig:Eoo}, and suppose that we can act by~$W$ on~$\h{Q}$,
with the action of \S\ref{subsec:action}, i.e., that $\h{Q}\subseteq
D$ (this is not true in general). Thus, as there are no limit points
in the red arc which is outside the triangle $\conv(\Delta)$, there
are also no limit points on its image by~$s_\alpha$, which is the
smaller red arc on the bottom left; and not either in the subsequent
image by~$s_\beta$. So $E$ seems to be contained in a fractal
self-similar subset $F$ of $\h{Q}$, obtained by removing from $\h{Q}$
all these iterated arcs. The rank $4$ pictures are even more
convincing of this property; see in particular Figures~\ref{fig:E3doo}
and~\ref{fig:sage}(b), where the fractal $F$ obtained (an ellipsoid
cut out by an infinite number of planes) looks like an Appolonian
gasket drawn on $\h{Q}$.

We conjecture that $E$ is actually equal to $F$. In the particular case where $\h{Q}$ is contained in $\conv(\Delta)$---e.g.,
Figure~\ref{fig:E3d3}, and all the cases of a rank $3$ Coxeter group with classical representation, as Figures~\ref{fig:sage}(a)
and~\ref{fig:E237}---, this means that $E$ fills up $\h{Q}$. We also conjecture the following more precise property:

\begin{conj}
  We say that $\Delta_I\subseteq \Delta$ is {\em generating} if
  $\h{Q_I}:=\h Q\cap \Span(\Delta_I)$ is included in
  $\conv(\Delta_I)$. Then we have the following properties:
  \begin{enumerate}[(i)]
  \item if $\Delta_I\subseteq \Delta$ is generating, then $E(\Phi_I)
    =\h{Q_I}$;
  \item the set $E$ is the topological closure of the fractal
    self-similar subset $F_0$ of $\h Q$ defined by:
    \[F_0:=W\cdot\Bigg( \bigcup_{\substack{\Delta_I\subseteq \Delta
      \\ \Delta_I \ \mathrm{generating}}} \h Q_I\Bigg).
    \]
  \end{enumerate}
\end{conj}

\begin{ex}
In the example pictured in Figure \ref{fig:E3doo}, the set $F_0$ defined above is the fractal constituted by the infinite union of all
the  circles. The set $F$ (and, conjecturally, the set $E$) is the complement, on the red sphere, of the union of the associated open ``disks''.
\end{ex}

%
%

\section{Construction of a dense subset of the limit roots from dihedral reflection subgroups}
\label{sec:dense}

As in the previous sections, we fix an infinite based root system
$(\Phi,\DD)$ in $(V,B)$ with associated Coxeter system $(W,S)$,
assuming that $\Delta$ is a basis for $V$, and we denote by
$E=E(\Phi)$ the accumulation set of its normalized roots. In this
section, we construct a nice countable subset of $E$, easy to
describe, and we show its density in $E$.

%

\subsection{Main Theorem: $E$ is obtained from dihedral reflection subgroups}
~

Consider two distinct positive roots $\rho_1$ and $\rho_2$, and define the dihedral reflection subgroup
$W':=\left<s_{\rho_1}, s_{\rho_2}\right>$ and
\[
\Phi':= \{\rho \in \Phi ~|~ s_{\rho} \in W' \} \ .
\]
>From Proposition~\ref{prop:dihedral}, we know that $\Phi'$ is a based
root system of rank $2$, associated to the dihedral group $W'$. Denote
by $\alpha, \beta$ its simple roots, and suppose
that~$B(\alpha,\beta)\leq -1$, i.e., that $W'$ is infinite and the
plane $\Span(\rho_1,\rho_2)$ intersects~$Q\setminus \{0\}$ (see
Proposition~\ref{prop:dihedral}~(ii)(a)). Similarly, the line
$L(\h{\rho_1},\h{\rho_2})$ intersects~$\h{Q}$ inside~$V_1$. Denote by
$u,v$ the elements of the intersection $L(\h{\rho_1},\h{\rho_2})\cap
\h{Q}$ (where $u=v$ if and only if ${{B(\alpha,\beta)=-1}}$). Because
of the rank $2$ picture (see Examples \ref{ex:rk2} and
\ref{ex:rk2order} and \S\ref{subsec:examplesE}), we know that the set
$E(\Phi')$ of limit points of normalized roots of $W'$ is equal to the
finite set $\{u,v \}$. This leads to the following natural definition.

\begin{defi}
  Let $E_2$ be the subset of $E$ formed by the union of the
  sets~$E(\Phi')$, where~$(\Phi',\Delta')$ is any based root subsystem
  of rank $2$ in $(\Phi,\Delta)$. Equivalently,
  \[
  E_2:= \bigcup_{\rho_1, \rho_2 \in \pp } L(\h{\rho_1},\h{\rho_2})\cap \h{Q} \ ,
  \]
  where $L(\h{\rho_1},\h{\rho_2})$ denotes the line containing
  $\h{\rho_1}$ and $\h{\rho_2}$.
\end{defi}

\begin{figure}[!ht]
\begin{minipage}[b]{.49\linewidth}
\centering
\captionsetup{width=\textwidth}
\scalebox{0.60}{



}

\caption{Geometric construction of~$E_2$, for the based root system with labels $4,4,4$ (roots of depth $\leq 5$ in blue, some
elements of~$E_2$ are in yellow diamonds).}
\label{fig:E'}
\end{minipage}
\hfill
\begin{minipage}[b]{.45\linewidth}
\centering
\captionsetup{width=\textwidth}
\scalebox{1}{
%

%

}
\caption{Construction of a sequence of elements of~$E_2$ ($x_n$, in yellow diamonds) from a sequence  of elements of~$\hpp$
($u_n$, in blue), both converging to~${\ell \in E}$.}
\label{fig:proof}
\end{minipage}
\end{figure}

Note that $E_2$ is countable, and geometrically easier to describe
than the whole set $E$. Indeed, it is constructed from $\Phi$ with
intersections rather than with limits. Moreover, $E_2=E$ when
$(\Phi,\Delta)$ is affine, or of rank~$2$. Figure~\ref{fig:E'} gives
an example of construction of some points of $E_2$. Surprisingly, $E_2$ still carries all the information of $E$, as implied by the following theorem.

\begin{theo} \label{thm:density}
Let $\Phi$ be an (infinite) based root system, $E$ the set of limit points of its normalized roots. Then the union $E_2$ of all limit
roots arising from dihedral reflection subgroups is dense in $E$.
\end{theo}

\begin{remark}
  Since $E$ is a limit set, it is closed, so the theorem implies that
  $E$ is the closure of $E_2$. Note also that if we define similarly
  $E_k$ as the subset of limit points arising from rank~$k$ reflection
  subgroups (with $k \geq 2$), then $E_k$ is dense in~$E$ too (since
  it contains $E_2$). However, this is not true if we consider only
  parabolic subgroups. For instance, any proper parabolic subgroup of
  an irreducible affine Coxeter group is finite and therefore does not
  provide any limit point (contradicting Corollary~\ref{rem:affine}).
\end{remark}

It turns out that an even smaller subset than $E_2$ is enough to
recover $E$. Define~$E^\circ_2$ by using the lines passing through a
normalized root and a \emph{simple} root:
\[
E_2^\circ:= \bigcup_{\substack{\alpha \in \Delta \\\rho \in \pp} } L(\alpha,\h{\rho})\cap \h{Q} \subseteq E_2 .
\]

\begin{prop}
The set $E_2$ is the $W$-orbit of $E_2^\circ$ for the action defined
in \S\ref{subsec:action}:
\[
E_2=W\cdot E_2^\circ.
\]
\end{prop}

\begin{proof}
Let $x\in E_2$ and $\rho_1,\rho_2\in\Phi$ such that $x\in L(\h{\rho_1},\h{\rho_2})$. Since $\Phi=W(\Delta)$, there is $w$ in $W$
such that $\alpha:=w(\rho_1)\in \Delta$. By Proposition~\ref{prop:actiongeo}~(ii), we have that~${w\cdot x\in  L(\h{\alpha},\h{w(\rho_2)})\cap Q}$, which proves that $w\cdot x\in E_2^\circ.$
\end{proof}

Theorem~\ref{thm:density} is in fact a direct consequence of the following stronger property.

\begin{theo} \label{thm:density0}
Let $\Phi$ be an (infinite) based root system, $E$ its set of limit points of normalized roots, and $E_2^\circ$ the subset of $E$
defined above. Then $E_2^\circ$ is dense in $E$.
\end{theo}

The rest of this section is devoted to the proof of
Theorem~\ref{thm:density0}. In regard to Proposition~\ref{prop:irred},
we assume that the based root system $(\Phi,\Delta)$, and the Coxeter
system $(W,S)$, are irreducible.

Consider an element $\ell$ of $E$. By definition, there exists a sequence of normalized roots $(u_n)_{n\in \mathbb{N}} \subseteq \hpp$
converging to $\ell$. We want to construct, from $(u_n)$, a new sequence of elements of $E_2^\circ$ converging to $\ell$ as well.
The geometric idea is quite simple, as illustrated in Figure~\ref{fig:proof}: we construct $x_n$ as the ``right'' intersection point of
the line $L(\alpha,u_n)$ with $Q$, for an $\alpha \in \Delta$ well chosen. The core of the proof is to show that it is always possible
to find such a simple root that makes the construction work.

We proceed in two steps, detailed in the next subsections.

%

\subsection{First step of proof: when the limit root is outside the radical}
~

Denote by $\Vp=\{v\in V\,|\, B(v,u)=0,\ \forall u \in V\}$ the radical of the bilinear form~$B$.

\begin{prop}  \label{prop:EVperp}
Let $\ell \in E$. If $\ell$ is not in the radical $\Vp$, then $\ell$ is in the closure of ${E_2^\circ}$.
\end{prop}

This proposition is a consequence of the following construction.

\begin{lemma}
  Let $(\rho_n)_{n\in \mathbb{N}}$ be a sequence in $\Phi^+$, such that $\h{\rho_n}$ converges to a limit~$\ell$ in $E$. Suppose that
  there exists $\alpha$  in $\DD$, such that $B(\alpha,\ell) \neq 0$. Then, there exists  a   sequence~$(x_n)_{n\in \mathbb N}$
  converging to $\ell$, such that, for any $n$ large enough, $x_n$ lies in the intersection~$L(\alpha,\h{\rho_n})\cap Q$.
\end{lemma}

If $\ell \in E$ and $\ell \notin \Vp$, then we are in the hypothesis of the lemma since $\Delta$ spans~$V$. Moreover, by definition
of $E_2^\circ$, the constructed sequence $(x_n)$ eventually lies in~$E_2^\circ$. Thus, Proposition~\ref{prop:EVperp} follows
directly from the lemma, which we prove now.

\begin{proof}
Since $B(\alpha,\ell)\neq 0$, we get from Proposition~\ref{prop:actiongeo}~(ii) that $L(\alpha,\ell)\cap Q$ contains two points (see
Figure~\ref{fig:proof}). Consider an element $x\in V_1\setminus \{\alpha\}$, and let us find the points of $L(\alpha,x)\cap Q$. Take an element
$u_\lambda$ in the line $L(\alpha,x)$
\[
u_{\lambda}=\lambda \alpha + (1-\lambda)x, \textrm{ with }\lambda \in \mathbb{R}.
\]
Then a quick computation gives
\begin{equation}
 q(u_\lambda)= 0 \quad \Leftrightarrow \quad q(\alpha -x) \lambda^2
 +2B(x,\alpha - x)\lambda + q(x) = 0 \label{eq:binome}
 \end{equation}
 (where $q$ is the quadratic form of $B$). For $x=\ell$, we know that
Equation~\eqref{eq:binome} has two solutions for $\lambda$. So by
continuity, there exists a neighbourhood $\Omega_\ell$ of $\ell$ in
$V_1$, such that, for any $x\in \Omega_\ell$,
Equation~\eqref{eq:binome} has two solutions as well (and then
$q(\alpha-x)\not = 0$).  Computing these solutions in function of $x$,
we find, after simplification,
\[
\lambda_\pm (x)= \frac{q(x)-B(\alpha,x) \pm \sqrt{B(\alpha,x)^2 - q(x)}}{q(\alpha-x)} \ ,
\]
and the two intersection points of $L(\alpha,x)$ with $Q$ (for $x\in \Omega_\ell$) are given by $u_+ (x)$ and $u_- (x)$, where
\[
u_{\eps}(x)=\lambda_\eps (x) \alpha + (1-\lambda_\eps (x))x,\ \text{for }\eps=+ \text{ or } -.
\]
For $x=\ell$, the intersection points are obviously $\ell$ and $s_\alpha\cdot \ell$; so, in regard to Proposition~\ref{prop:visible},
either $u_{+}(\ell)=\ell$ if $B(\alpha,\ell)>0$ (that is, when $\ell$ is visible from~$\alpha$), or $u_{-}(\ell)=\ell$ if $B(\alpha,\ell)<0$
(that is, when $s_\alpha \cdot \ell$ is visible from $\alpha$). Define, for~${x\in \Omega_\ell}$,
\[
f_\alpha(x)= u_\eps (x)= \lambda_\eps (x) \alpha + (1-\lambda_\eps (x)) x\ ,
\]
where $\eps=+$ if $B(\alpha,\ell)>0$, and $\eps=-$ if $B(\alpha,\ell)<0$. Then,
\begin{itemize}
\item $f_{\alpha}(\ell)=\ell$;
\item $f_{\alpha}$ is a continuous map on $\Omega_\ell$;
\item for all $x\in \Omega_\ell$, $f_{\alpha}(x) \in L(\alpha,x)\cap Q$.
\end{itemize}

\noindent Now consider a sequence $(\rho_n)_{n\in \mathbb{N}}$ in $\pp$, such that $\h{\rho_n}$ converges to $\ell$. For $n$ large
enough, $\h{\rho_n}\in \Omega_{\ell}$, and we set $x_n = f_{\alpha}(\h{\rho_n})$. Then $x_n \in L(\alpha,\h{\rho_n})\cap  Q$,
and by continuity of $f_{\alpha}$, $x_n$ converges to $f_{\alpha}(\ell)=\ell$.
\end{proof}

%

\subsection{Second step of proof: when the limit root is inside the radical}
~

To finish the proof, we only need to deal with the case where $\ell$
is in the radical~$V^\perp$ of the form $B$. Since $\ell$ is also in
$\cone(\DD)$, the following proposition, which is the same as~\cite[Lemma 6.1.1]{krammer:conjugacy},
 implies that this case can only happen when $\Phi$ is an affine based root system.

\begin{prop} \label{prop:noninter}
Let $\Phi$ be an (infinite) irreducible based root system. If $\Phi$ is not affine, then
\[
\Vp \cap \cone(\Delta) = \{0\} \ .
\]
\end{prop}

The proof of this statement is an application of the Perron-Frobenius theorem; we refer to \cite[Lemma 6.1.1]{krammer:conjugacy} for details.
\smallskip

So we focus on the last remaining case:   $\Phi$ is an irreducible affine root
system. In this case, by Corollary~\ref{rem:affine}, $E$ is a
singleton, so it is sufficient to check that $E_2^\circ$ is
nonempty. But $E_2^\circ$ is empty if and only if $E_2$ is empty, if
and only if any reflection subgroup of $W$ of rank $2$ is finite; and
this cannot happen because of the following classical property.

\begin{lemma} \label{lem:rk2finite}
Let $W$ be a Coxeter group. If every reflection subgroup of rank $2$ of~$W$ is finite, then $W$ is finite.
\end{lemma}

\begin{proof}
We use the concept of small roots, described for instance in
\cite[\S4.7]{bjorner-brenti}; all the references in this proof are to
this book.  If every reflection subgroup of rank~$2$ of~$W$ is finite,
then any covering edge of the root poset on $\pp$ (see Def.~4.6.3) is
a short edge (see \S4.7), and any root is a small root. But it is
known that the set of small roots is finite (Theorem~4.7.3, taken from
Brink-Howlett \cite{brink-howlett}), so $\Phi$ is finite and~$W$ must
be finite as well.
 \end{proof}

\noindent This concludes the proof of Theorem~\ref{thm:density0} and therefore the proof of Theorem~\ref{thm:density}.

%
%

\section{How to construct the limit roots for reflection subgroups}  \label{sec:subgps}

In this section, we explain how to slightly extend the definition of $E$ so that it holds for a reflection subgroup $W'$ of a Coxeter
group $W$. To do so we extend the definition of limit roots such that it applies even when the set of simple roots is not a
basis. In the last subsection, we also detail the relations between the set of limit roots associated to a reflection subgroup of $W$
(especially to a parabolic subgroup) and the set $E$ associated to $W$.

%

\subsection{Root subsystems of a based root system} \label{subsec:subsystems}
~

First, we define the notion of a root subsystem. Let~$(\Phi,\Delta)$ be a based root system of $(V,B)$ with associated Coxeter
system~$(W,S)$, according to Definition~\ref{def:root}. Consider a finitely generated reflection subgroup $W'$ of $W$, together
with its canonical set $S'$ of Coxeter generators, as explained in \S\ref{subsec:reflsg}. There is a natural way to construct a root
system for $W'$ from the one for $W$. If
\[
\Phi':=\{\rho \in \Phi \, | \, s_\rho \in W'\} \quad\text{and}  \quad \Delta':=\{\rho \in \Phi \, | \, s_\rho \in S' \},
\]
then $(\Phi',\Delta')$ is a based root system of $(V,B)$ with associated Coxeter system~$(W',S')$ (see~\cite[Lemma~3.5]{bonnafe-dyer}
for details). Note that $\Delta'\subseteq \Phi^+$ does not have to intersect~$\Delta$.  We call $(\Phi',\Delta')$ a {\em based root
subsystem of~$(\Phi,\Delta)$}. In \S\ref{subsec:examplesE}, we already mentioned the specific case of \emph{standard parabolic}
root subsystems, which correspond to standard parabolic Coxeter subgroups of $W$.

The simplest definition of $E(\Phi')$ in this setting is to consider the subset of $E(\Phi)$ consisting of all the points which are limits
of sequences in $\h{\Phi'}$. A natural question arises: is this definition consistent with the intrinsic one, when we consider
~$(\Phi',\Delta')$ as a root system by itself? The first problem here is that, in \S\ref{sec:limit}, we only defined~$E$ when the set
of simple roots is linearly independent, and this is not always the case for $(\Phi',\Delta')$, as illustrated in the following example.

\begin{ex} \label{ex:HigherRank}
Consider the Coxeter group $W$ of rank $3$ with ${S=\{s,t,u\}}$ and
$m_{s,t}=m_{t,u}=m_{s,u}=4$ (whose Coxeter diagram is on
Figure~\ref{fig:E'}). Let $(V,B)$ be the classical geometric
representation of $W$. Consider the reflection subgroup $W'$ generated
by $S'=\{s,u,tst,tut\}$. It is easy to check that its simple system
$\Delta'$ is~$\{\alpha_s,\alpha_u,t(\alpha_s),t(\alpha_u)\}$ and
satisfies the conditions of Definition~\ref{def:root} (Condition~(i) is
critical here). So $(W',S')$ is a Coxeter group of rank $4$ and $V$ is
a geometric~$W'$-module of dimension $3$ with based root system
$(\Phi',\Delta')$ where $\Phi'$ is the set of roots~$\rho$ in~$\Phi$
such that~$s_\rho \in W'$.  Consider now the matrix
$A=(B(\alpha_x,\alpha_y))_{x,y\in S'}$, written accordingly to
$(\alpha_s,\alpha_u,t(\alpha_s),t(\alpha_u))$:
\[
A=\left(
\begin{array}{cccc}
1&-\frac{\sqrt 2}{2}&0&-1-\frac{\sqrt 2}{2}\\
-\frac{\sqrt 2}{2}&1&-1-\frac{\sqrt 2}{2}&0\\
0&-1-\frac{\sqrt 2}{2}&1&-\frac{\sqrt 2}{2}\\
-1-\frac{\sqrt 2}{2}&0&-\frac{\sqrt 2}{2}&1
\end{array}
\right).
\]
This matrix satisfies Equation~\eqref{equ:matrice}, so we can construct (as in \S\ref{subsec:georep2}) the canonical geometric
$W$-module $(V_A,B_A)$, of dimension $4$, associated to $W'$. Observe that~$(V_A,B_A)$ is not the classical geometric
representation of $W'$.
\end{ex}

Such examples explain why a satisfying definition of a based root
system should not require the simple roots to form a basis, but only
to be positively independent (Condition~(i) in
Definition~\ref{def:root}). They also raise the question of the
possible relations between $(V,B)$ and $(V_A,B_A)$ as $W'$-modules
(see \S\ref{subsec:Wmodules}).

%

\subsection{How to cut the rays of roots in general: hyperplanes transverse to~\texorpdfstring{$\Phi^+$}{Phi+}} \label{sse:transverse}
~

In \S\ref{sec:limit} we constructed the set $E$ of limit roots from \emph{normalized} roots, by cutting the rays of positive roots
with the hyperplane $V_{1}$. To be able to define properly~$E$ in the case where $\Delta$ is not necessarily a basis, we need
to find analogues for the cutting hyperplane $V_1$ and the normalization map $\h{\cdot}$.

Even when $\Delta$ is a basis, it is natural to ask the following question: since $\h{\Phi}$ is an affine representative of
$\mathbb P\Phi$ obtained by ``cutting'' $\Phi$ by a particularly well-chosen affine hyperplane, what happens if we change the
affine hyperplane, i.e., the representative of $\mathbb P\Phi$? For instance, in the case of an affine dihedral group (see
Example~\ref{ex:affine}), if we cut the rays of $\Phi$ by an affine line directed by $\alpha+\beta$ (thus parallel to the radical of
$B$), then the points representing the roots are not all contained in the convex hull of the points representing the simple roots,
and these normalized roots diverge to infinity. This fact suggests the following general definition, that also takes care of the case
where $\Delta$ is not a basis (but still positively independent).

\begin{defi} \label{def:transverse}
Let $(\Phi,\Delta)$ be a based root system in $(V,B)$. We call an affine hyperplane~$H$ {\em transverse to $\Phi^+$} if the
intersection of $H$ with any ray $\mathbb R^{+}\alpha$ directed by a simple root $\alpha\in \Delta$ is a nonzero point, i.e., if
$\mathbb R^{+*}\alpha\cap H$ is a singleton for all~$\alpha\in\Delta$.
\end{defi}

If $\Delta$ is a basis for $V$, the affine hyperplane $V_1$ we used to cut the rays of $\Phi$ is obviously transverse.  In the case
where $\Delta$ is only positively independent, the definition of $V_1$ does not work (it could be equal to $V$). But transverse
hyperplanes still exist. Consider a linear hyperplane $H_0$, separating $\Phi^+$ and $\Phi^-$: it exists because $\Delta$ is
positively independent (see Remark~\ref{rk:positively}). Then the affine hyperplane~${H:=\sum_{\alpha\in\Delta} \alpha +H_0}$
is transverse.

Let $H$ be a transverse hyperplane directed by a hyperplane $H_0$. Let $H_0^+$ be the open halfspace supported by $H_0$
and containing $H$ ($H$ is parallel to $H_0$ and $0$ cannot be in $H$ by definition). So $H_0^+$ contains also $\Delta$ and
therefore $\Phi^+$, which is a subset of $\cone(\Delta)$. In particular, we obtain easily the following statement: an affine
hyperplane $H$ is transverse if and only if $\mathbb R^{+*}\beta\cap H$ is a singleton for all $\beta\in\Phi^+$.

Like for $V_1$, a normalization map can be applied to~$\Phi$:
\[
\begin{array}{ccl}
  \pi_H:V\setminus H_0 & \to & H 
\end{array}
\]
where $\pi_H(v)$ is the intersection point of $\mathbb R v$ with
$H$. For instance, when $\Delta$ is a basis,~${\pi_{V_1}(v)=\h v}$.
Note that since $H$ is transverse, $\pi_H(\Phi)$ is contained in the
polytope $\conv(\pi_H(\Delta))$.  Denote by $E(\Phi,H)$ the set of
accumulation points of $\pi_H(\Phi)$. The following straightforward
proposition states that the topology of our object of
study,~$E(\Phi)$, does not depend on the choice of the transverse
hyperplane chosen to cut the rays of~$\Phi$, as suggested by the point
of view of projective geometry.

\begin{prop} \label{prop:transverse}
  Let $(\Phi,\Delta)$ be a based root system, and let $H$, $H'$ be two
  hyperplanes, transverse to $\Phi^+$. Then $\pi_H$ induces a
  homeomorphism from $\conv(\pi_{H'}(\Delta))$ to
  $\conv(\pi_H(\Delta))$, whose inverse is the restriction of
  $\pi_{H'}$. Moreover, $\pi_H$ maps (bijectively) $\pi_{H'}(\Phi)$ to
  $\pi_{H}(\Phi)$ and $E(\Phi,H')$ to $E(\Phi,H)$.
\end{prop}

\begin{proof}
  For $v\in V$, $\pi_H(v)$ and $\pi_{H'}(v)$ (when they are defined) are
  colinear to $v$. As~$\Delta$ is contained in $V \setminus H_0$, so
  are $\pi_{H'}(\Delta)$ and $\conv(\pi_{H'}(\Delta))$. So, the
  mentioned restrictions of $\pi_H$ and $\pi_{H'}$ are, by
  construction, continuous, mutual inverses, and swap $\pi_{H'}(\Phi)$
  and $\pi_{H}(\Phi)$. Consequently, they also swap the sets of limit
  points of these two sets, i.e., $E(\Phi,H')$ and $E(\Phi,H)$.
\end{proof}

\begin{remark}
  The geometric action of $W$ defined in \S\ref{subsec:action} does
  not depend on the choice of the transverse hyperplanes employed to
  cut $\Phi$: if $H$ and $H'$ are two hyperplanes that are transverse
  to $\Phi^+$, then the map $\pi_{H'}$ induces by restriction an
  isomorphism of $W$-sets from $\pi_H(\Phi)\sqcup E(\Phi,H)$ to
  $\pi_{H'}(\Phi)\sqcup E(\Phi,H')$.
\end{remark}

\begin{remark}
  Changing the transverse hyperplane can still change the geometry
  overall. For instance, consider a rank $3$ based root system with
  $B$ of signature~$(2,1)$, so that $Q$ is a circular cone. Let $H$ be
  a transverse affine plane. In general, the intersection $Q\cap H$ of
  $H$ with the isotropic cone $Q$ can be an ellipse, a parabola or an
  hyperbola. In the case of an hyperbolic cut, only one branch
  intersects $\conv(\pi_H(\Delta))$. In a following paper~\cite{dhr},
  we prove and use the fact that for a hyperbolic group of rank $3$,
  there always exists a hyperplane $H$, transverse to~$\Phi^+$, such
  that $Q\cap H$ is a circle (the analogous property is actually valid
  for hyperbolic groups of higher rank).
\end{remark}

For now, we apply this result to explain in the following subsection that all the results of \S\S\ref{sec:limit}-\ref{sec:action}-\ref{sec:dense}
transfer easily to the general case where $\Delta$ is not necessarily a basis.

%

\subsection{From the linearly independent case to the positively independent case}   \label{subsec:Wmodules}
~

In the case where $\Delta$ is not a basis, there is another way to study the limit points of roots (without using general transverse
hyperplanes), by pulling back all the structure to a space where $\Delta$ \emph{is} a basis, using the canonical geometric
$W$-module defined in \S\ref{subsec:georep2}. let us first give the following definition (motivated by Example~\ref{ex:HigherRank}).

\begin{defi}
Let $(W,S)$ be a Coxeter group and $A=(a_{s,t})$ be a matrix that satisfies Equation~\eqref{equ:matrice}. A pair $(V,B)$ is a
{\em geometric $W$-module associated to $A$} if
\begin{enumerate}[(i)]
\item $V$ is a finite dimensional real vector space containing a subset $\Delta={\{\alpha_s\,|\,s\in S\}}$ that is positively independent;

\item $B$ is a bilinear form such that $B(\alpha_s,\alpha_t)=a_{s,t}$ for all $s,t\in S$.
\end{enumerate}
In this case, since $\Delta$ satisfies the requirement of Definition~\ref{def:root},  $(V,B)$ is a geometric $W$-module (as in
Definition~\ref{def:root}); the associated based root system is called the {\em based root system associated to $(V,B)$}.
\end{defi}

Note that any geometric $W$-module $(V,B)$ of Definition~\ref{def:root} is a geometric $W$-module associated to a certain
matrix $A$ determined by Conditions~(ii)~and~(iii) of Definition~\ref{def:root}.

Obviously, for each matrix $A$ satisfying Equation~\eqref{equ:matrice}, the simplest geometric $W$-module associated to $A$ is
the ``canonical'' geometric $W$-module $(V_A,B_A)$ (defined in \S\ref{subsec:georep2}), where we declare $\Delta$ as a basis.
However, as Example~\ref{ex:HigherRank} shows, the set of geometric modules $(V,B)$---up to isomorphism---associated to
$A$ may contain more than the class of $(V_A,B_A)$. Still, the canonical geometric $W$-module associated to~$A$ carries all the
combinatorics of the based root system of~$A$, as described in the following proposition.\footnote{This proposition was suggested
by M.~Dyer (personal communication, September 2011), see also \cite[Proposition 2.1]{fu1}.}

\begin{prop} \label{prop:bijection}
Let $(W,S)$ be a Coxeter system and $A$ be an associated matrix satisfying Equation~\eqref{equ:matrice}. Let $(V_A,B_A)$ be the
canonical geometric $W$-module, and~$(V,B)$ be a geometric $W$-module associated to $A$, with based root system $(\Phi,\Delta)$.
Then there is a morphism of $W$-modules $\varphi_A:V_A\to V$, preserving $B$, that restricts to a bijection from $\Phi_A$ to~$\Phi$,
such that~$\varphi_A(\Delta_A)=\Delta$ and $\varphi_A(\Phi_A^+)=\Phi^+$.
\end{prop}

\begin{proof}
Write $\Delta=\{\alpha_s\,|\, s\in S\}$ and $\Delta_A=\{\gamma_s\,|\, s\in S\}$. So
\[
B(\alpha_s,\alpha_t)=a_{s,t}=B_A(\gamma_s,\gamma_t).
\]
Now let
\[
\varphi_A: V_A\to V
\]
be the linear map defined by mapping any element $\gamma_s$ of the basis $\Delta_A$ of $V_A$ to $\alpha_s$. We have
$B\big(\varphi_A(\gamma_s),\varphi_A(\gamma_t)\big)=B(\alpha_s,\alpha_t)=B_A(\gamma_s,\gamma_t)$; so that it is easy to check
that $\varphi_A$ is a morphism of $W$-modules, preserving $B$.

\noindent The fact that $\varphi_A$ restricts to a bijection from $\Phi_A$ to $\Phi$ is nontrivial, see \cite[Proposition 2.1]{fu1}.
Finally, as by construction we have $\varphi_A(\Delta_A)=\Delta$, we get directly that $\varphi_A$ also maps $\Phi_A^+$ to $\Phi^+$.
\end{proof}

Proposition~\ref{prop:bijection} tends to imply that in the previous sections we did not lose much generality by studying only root
systems where $\Delta$ is a basis.

Indeed, let $(V,B)$ be a geometric $W$-module with based root system $(\Phi,\Delta)$, where $\Delta$ is not necessarily linearly
independent. Then we know that $(V,B)$ is a geometric $W$-module associated to a matrix $A$; but we can also construct,
from \S\ref{subsec:georep2}, the canonical geometric $W$-module $(V_A,B_A)$ and the based root system $(\Phi_A,\Delta_A)$
for which~$\Delta_A$ is a basis. The proposition above tells us that the map $\varphi_A$ is a morphism of $W$-modules carrying
bijectively the combinatorics of the root system $\Phi_A$ to that of the root system $\Phi'$.  Moreover, since $\varphi_A$ is a linear map,
it is continuous, and the topological properties we established in $V_A$ are carried to $V$ through $\varphi_A$.

In~$V_A$, the simple system $\Delta_A$ is a basis, so we can apply our usual construction to $(\Phi_A,\Delta_A)$, building
the hyperplane $(V_A)_1$, the normalized roots $\widehat{\Phi_A}$ and the limit roots $E(\Phi_A)$. Now choose a hyperplane
$H$ in $V$, transverse to $\Phi^+$. Thanks to Proposition~\ref{prop:bijection}, we see that
\[
{{\conv\big(\pi_H(\Delta)\big)=\pi_H\circ \varphi_A\big(\conv(\Delta_A)\big)}}\ \textrm{  and }\ {{\pi_H(\Phi)=\pi_H\circ\varphi_A(\widehat{\Phi_A})}}.
\]
Since $\varphi_A$ and $\pi_H$ are continuous, $\pi_H\circ\varphi_A$ is continuous. Therefore we obtain:
\[
E(\Phi,H)=\pi_H\circ\varphi_A\big(E(\Phi_A)\big).
\]
Thus, the sets of normalized roots and limit roots of a based root system $(\Phi,\Delta)$ attached to a matrix $A$ are the images
of the same objects for the based root system~$(\Phi_A,\Delta_A)$. Overall, this means that all the properties we proved in the
previous sections in the case where $\Delta$ is a basis are simply transferred to the general case.

However it is not clear whether the map $\pi_H\circ\varphi_A$ induces a bijection between~$E(\Phi_A)$ and $E(\Phi,H)$: we
do not know if in general $E(\Phi,H)$ is a faithful image of~$E(\Phi_A)$; but $E(\Phi,H)$ is always a projection of $E(\Phi_A)$
by construction.

%

\subsection{On limit roots for reflection subgroups} \label{sse:reflectionsg}
~

We now address how to define the limit roots for a reflection
subgroup. Let~$({\Phi,\Delta)}$ be a based root system of $(V,B)$ with
associated Coxeter system $(W,S)$. Fix $H$ a hyperplane that is
transverse to $\Phi^+$, according to Definition~\ref{def:transverse}
(e.g., take~${H=V_{1}}$ if $\Delta$ is a basis). Let $(\Phi',\Delta')$
be a based root subsystem of~$(\Phi,\Delta)$. Then:
\begin{itemize}
\item The hyperplane $H$ is also transverse to $\Phi'^+$.
\item So, we define intrinsically the set $E(\Phi',H)$, constituted
  with the limit points of $\pi_{H}(\Phi')$.
\item Obviously $E(\Phi',H)$ is a subset of $E(\Phi,H)$.
\end{itemize}

So, as wanted in \S\ref{subsec:subsystems}, we get an intrinsic
definition of the limit roots for a root subsystem, which is
compatible with the inclusion inside the whole set of limit
roots. Moreover, thanks to \S\ref{subsec:Wmodules}, if we set
\begin{itemize}
\item $A$ the matrix associated to the root subsystem $(\Phi',\Delta')$,
\item $(V_A,B_A)$ the canonical geometric $W'$-module for $A$, with
  based root system $(\Phi_{A},\Delta_{A})$,
\item $\varphi_A:V_A \to V$ the associated map defined in
  Proposition~\ref{prop:bijection},
\item $E(\Phi_A)$ the set of limit roots for $(\Phi_{A},\Delta_{A})$,
\end{itemize}
then $E(\Phi',H)$ is the renormalization of the image under
$\varphi_A$ of $E(\Phi_A)$
\[
E(\Phi',H)=\pi_{H}(\varphi_A(E(\Phi_A))).
\]

In the following, to lighten the notations, the cutting hyperplane $H$ is implicit.

The set $E(\Phi')$ is always included in $\conv(\Delta')$ (and in $E(\Phi)$). One can wonder whether the inclusion
$E(\Phi')\subseteq E(\Phi)\cap \conv(\Delta')$ is an equality. Consider the most simple case, where the reflection subgroup
$W'$ is in fact a standard parabolic subgroup. Let~$I$ be a subset of~$S$, $\Delta_I=\{\alpha_s\,|\,s\in I\}$ and $\Phi_I$ its
orbit under $W_I=\langle I\rangle$ (i.e., a parabolic root subsystem). Denote by $F_I$ the convex hull of $\Delta_I$, and
$E(\Phi_I)$ the set of accumulation points of $\h{\Phi_I}$, which lives inside $F_I$. In this case, the question is: if a  limit root
lies in the face $F_{I}$ of the simplex $\conv(\Delta)$, is this a limit point of normalized roots of $W_{I}$? Surprisingly enough,
the answer can be negative, as shown by the following counterexample.

\begin{ex} \label{ex:cex}
Take the rank $5$ root system $\Phi$ with
$\DD=\{\alpha,\beta,\gamma,\delta,\eps\}$ and the labels
$m_{\alpha,\beta} = m_{\delta,\eps}=\oo$, $m_{\beta,\gamma} =
m_{\gamma,\delta}=3$, and the others equal to $2$.  Take~$\Delta_I=\DD
\setminus \{\gamma\}$, so that $W_I$ is the direct product of two
infinite dihedral groups. Then we have
$E(\Phi_I)=\{\frac{\alpha+\beta}{2}, \frac{\delta + \eps}{2}\}$ (see
\S\ref{subsec:examplesE}). But if we consider
$\rho_n=(s_{\alpha}s_\beta s_\eps s_\delta)^n(\gamma)$, it is easy to
check that $\h{\rho_n}$ tends to $\frac{\alpha+\beta + \delta +
  \eps}{4}$, that lies in $E(\Phi)\cap F_I$ but not in $E(\Phi_I)$.
\end{ex}

In a subsequent paper \cite{dhr}, we show that this property holds nevertheless for the set $E_2$ studied in \S\ref{sec:dense},
i.e., $E_2(\Phi_I)=E_2(\Phi)\cap F_I$. We also define other natural smaller {\em dense} subsets of $E$ for which this property
of parabolic restriction works.

%

\subsection*{Acknowledgements.}

We would like to thank Matthew Dyer for having communicated to us his
preprint \cite{dyer:imaginary}, for many fruitful discussions at LaCIM
in September 2011, for the idea of the counterexample~\ref{ex:cex} and
the suggestion of Proposition~\ref{prop:bijection}. We are also
grateful to the referees of an extended abstract of this article
(submitted to the conference FPSAC 2012) for helpful comments.
Finally, we gratefully thank the anonymous referee for the careful reading of our manuscript, for the references to \cite{Sa91} and \cite[Lemma 6.1.1]{krammer:conjugacy} and for the valuable comments and suggestions, which contributed to improve the quality of this article.


\newcommand{\etalchar}[1]{$^{#1}$}
\def\cprime{$'$}

\end{document}